\documentclass[11pt]{dk_article}
\usepackage[top=3cm, bottom=3cm, left=3cm, right=3cm]{geometry}
\usepackage{graphicx}
\usepackage{todonotes}

\usepackage{color}
\usepackage{hyperref}
\definecolor{darkred}{rgb}{.7,0,0}

\definecolor{darkgreen}{rgb}{0,0.7,0}

\definecolor{darkblue}{rgb}{0,0,0.7}
\newcommand{\iid}{i.i.d.\,}

\newcommand{\yd}{y^{\dagger}}
\newcommand{\rd}{\rho^{\dagger}}
\newcommand{\ud}{u^{\dagger}}
\newcommand{\xid}{\xi^{\dagger}}
\newcommand{\etd}{\eta^{\dagger}}
\newcommand{\ztd}{\zeta^{(n)}}
\newcommand{\he}{\widehat{e}}
\newcommand{\mfC}{\mathfrak C}
\newcommand{\mfz}{\mathfrak z}

\newcommand{\tp}{\widetilde q}

\newcommand{\dtv}{d_{\rm TV}}

\newcommand{\ftilde}{\tilde{f}}
\newcommand{\vtilde}{\tilde{v}}

\newcommand{\murec}{{\mu}}

\newcommand{\unit}{\mathbb{I}}

\title{Ergodicity and Accuracy of Optimal Particle Filters\\ 
for Bayesian Data Assimilation} 
\author{D. Kelly$^\star$ and  A.M. Stuart$^\dagger$}
\institute{$^\star$ New York University, \email{dtbkelly@gmail.com} \\ $^\dagger$ California Institute of Technology, \email{astuart@caltech.edu}}

\newcommand{\vhat}{\widehat{v}}
\newcommand{\uhat}{\widehat{u}}

\newcommand{\CY}{\mathcal{Y}}
\newcommand{\CZ}{\mathcal{Z}}

\newcommand{\bbP}{\mathbb{P}}
\newcommand{\bbQ}{\mathbb{Q}}
\newcommand{\bbE}{\mathbb{E}}
\newcommand{\bbI}{\mathbb{I}}

\newcommand{\pchain}{z}
\newcommand{\pchaintilde}{{z}}

\newcommand{\CM}{\mathcal{M}}
\newcommand{\CX}{\mathcal{X}}

\newcommand{\bpfw}{\widehat{\rho}}
\newcommand{\bpfe}{\rho}

\newcommand{\opfw}{\widehat{\mu}}
\newcommand{\opfe}{\mu}

\newcommand{\sirw}{\widehat{\mu}}
\newcommand{\sire}{\mu}

\newcommand{\gopf}{\nu}

\begin{document}
\maketitle

\begin{abstract}
Data assimilation refers to 
the methodology of combining dynamical models and 
observed data with the objective of improving state estimation. Most data 
assimilation algorithms are viewed as approximations of the Bayesian 
posterior (filtering distribution) on the signal given the observations.
Some of these approximations are controlled, 
such as particle filters which may be 
refined to produce the true filtering distribution in the large particle number
limit, and some are uncontrolled, such as ensemble Kalman filter
methods which do not recover the true filtering distribution in the large
ensemble limit. Other data assimilation algorithms, such as cycled 3DVAR 
methods, may be thought of as controlled estimators of the state, in the
small observational noise scenario, but are also uncontrolled in general
in relation to the true filtering distribution. For particle filters and ensemble Kalman filters it is of practical 
importance to understand how and why data assimilation methods can be 
effective when used with a fixed small number of particles, since 
for many large-scale applications it is not practical to deploy 
algorithms close to the large particle limit asymptotic. 
In this paper we address this question for particle filters and, in particular, 
study their accuracy (in the small noise limit)
and ergodicity (for noisy signal and observation)
without appealing to the large particle 
number limit. We first overview the accuracy and minorization 
properties for the true filtering distribution, working in the setting of 
conditional Gaussianity for the dynamics-observation model. We then show that 
these properties are inherited by optimal particle filters 
for any fixed number of particles, and use the minorization
to establish ergodicity of the filters. For completeness
we also prove large particle number consistency results 
for the optimal particle filters, by writing the update 
equations for the underlying distributions as recursions. 
In addition to looking at the optimal particle filter with standard 
resampling, we derive all the above results for (what we term) the 
Gaussianized optimal particle filter and show that the 
theoretical properties are favorable for this method,
when compared to the standard optimal particle filter.
\end{abstract}

\section{Introduction}

\subsection{Background and Literature Review}

Data assimilation describes the blending of dynamical models with data, with the 
objective of improving state estimation and forecasts. The use of data 
assimilation originated in the geophysical sciences, but
is now ubiquitous in engineering and the applied sciences. In numerical weather 
prediction, large scale ocean-atmosphere models are assimilated with massive data 
sets, comprising observational data from satellites, ground based weather stations
and underwater sensors for example \cite{bauer2015quiet}. Data assimilation is 
prevalent in robotics; the SLAM problem seeks to use sensory data made by robots in an unknown environment to create a map of that environment and locate the robot within it \cite{thrun2002particle}. It is used
in modelling of traffic flow \cite{bayen}.
And data assimilation is being used in
bio-medical applications such as glucose-insulin systems \cite{sedigh2012data}
and the sleep cycle \cite{sedigh2012reconstructing}. These examples
serve to illustrate the growth in the use of the methodology, its breadth
of applicability and the very different levels of fidelity present in the
models and the data in these many applications.
\par
Although typical data assimilation problems can be understood from a Bayesian perspective, for non-linear and potentially high dimensional models it is often infeasible to make useful exact computations with the posterior. To circumvent this problem, practitioners have developed assimilation methods that approximate the true posterior, but for which computations are more feasible. In the engineering communities, particle filters have been developed for this purpose, providing empirical approximations of non-Gaussian posteriors \cite{doucet2000sequential,doucet2001introduction}. In the geoscience communities, methods are typically built on Kalman filtering theory, after making suitable Gaussian approximations \cite{LSZ15}; such methods include variational methods like 3DVAR and 4DVAR \cite{ghil1991data,lorenc1986analysis}, the extended Kalman filter (ExKF) \cite{evensen1992using} and the ensemble Kalman filter (EnKF) \cite{burgers1998analysis, evensen2003ensemble}. For these methods the underlying Gaussian
ansatz render them, in general, invalid as approximations of the true filtering
distribution \cite{law2014deterministic}.   
\par
Despite their widespread use, many of these algorithms used in geophysical
applications remain mysterious from a theoretical perspective. At the heart of the mystery is the fact that data assimilation methods are frequently and successfully implemented in regimes where the approximate filter is not provably valid; it is not known which features of the posterior (the true filtering distribution) are reflected in the approximation and which are not. For example, the ensemble Kalman filter is often implemented with ensemble 
size several orders of magnitude smaller than needed to reproduce large sample
size behaviour, and is applied to problems for which the Gaussian ansatz may not
be valid; it nonetheless can still exhibit skillful state estimates, with high correlations between the estimate and true trajectories \cite{houtekamer2005ensemble,majda2012filtering}. Indeed, the success of the methods in this non-asymptotic regime is the crux of their success; the methods would often be computationally 
intractable at large ensemble sizes.
\par
This lack of theory has motivated recent efforts to better understand the properties of data assimilation methods in the practical, non-asymptotic regimes. The 3DVAR algorithm has been investigated in the context of toy models for numerical weather prediction, including the Lorenz-63 \cite{law2012analysis}, Lorenz-96 \cite{law2014filter} and $2d$ Navier-Stokes equations \cite{law2013accuracy};
see also \cite{moodey2013nonlinear}. These works focus primarily on the question of accuracy -- how well does the state estimate track the true underlying signal. Accuracy for the EnKF with fixed ensemble size was first investigated in \cite{kelly2014well}; the study of accuracy for the EnKF was further developed in \cite{tong2016enkf} using linear models with random coefficients, but much more realistic (practical) assumptions on observations than \cite{kelly2014well}, and moreover focussing on covariance consistency through the Mahalanobis norm. The articles \cite{tong2016nonlinear,tong2015nonlinear} were the first to investigate the stability of EnKF with fixed ensemble size, by formulating the filter as a Markov chain and applying coupling techniques; here by stability we mean 
robustness with respect to initialization, and study the issue through the
lens of ergodicity. This line of research has been continued in \cite{del2016stability} by framing the EnKF as a McKean-Vlasov system. The limitations of the non-practical regimes have also been investigated; in \cite{kelly2015concrete} the authors construct simple dissipative dynamical models for which the EnKF is shown to be highly unstable with respect to initial perturbations. This was the first theoretical insight into the frequently observed effect of catastrophic filter divergence \cite{harlim2010catastrophic}. 
\par
For the nonlinear filtering distribution itself, there has been a great deal of research over the last several decades, particularly on the question of stability. Conditional ergodicity for the filtering distribution for general nonlinear hidden Markov models has been investigated in \cite{kunita1971asymptotic} and later refined in \cite{van2009stability}. Ergodicity for nonlinear filters has been discussed in \cite{kleptsyna2008discrete,douc2009forgetting,crisan2008stability} and exponential convergence results were first obtained in \cite{atar1997exponential,budhiraja1997exponential}.  
  
\par
\subsection{Our Contributions}
For particle filters, much of the theoretical literature focuses on the question of consistency in the large ensemble limit, that is, does the empirical approximation converge to the true posterior as the number of particles in the ensemble $N$ approaches infinity. However, in many high dimensional applications
such as robotics \cite{thrun2002particle} and ocean-atmosphere forecasting \cite{van2009particle}, 
particle filters are implemented in the non-asymptotic regime. Indeed
in the geosciences, new filtering algorithms have been proposed to beat the curse of dimensionality and are implemented with ensemble sizes many orders of magnitude smaller than the state dimension \cite{van2010nonlinear}. In this article 
we contribute to the program of analyzing filtering algorithms in practical 
small ensemble regimes, focusing on the accuracy and stability of particle 
filters for fixed ensemble sizes. In particular, we address the following 
question concerning the long-time behaviour of the particle filters: if it 
is known that the true posterior distribution is accurate and satisfies 
a minorization condition, can accuracy and conditional ergodicity be proved 
for the approximate filter?
\par
We focus our attention on the optimal particle filter (OPF) \cite{akashi1977random,zaritskii1975monte,liu1995blind}. The OPF is a sequential importance sampling procedure in which particle updates are proposed using a convex combination of the model prediction and the observational data at the next time step. For details on the OPF, including the justification for calling it optimal, see \cite{doucet2000sequential}. There are two main reasons that we focus our attention on the OPF. First, the optimal particle filter is known to compare favorably to the standard particle filter, particularly from the perspective of weight degeneracy in high dimensions \cite{snyder2011particle,snyder2015performance}. Indeed the optimal particle filter can be considered a special case of more complicated filters that have been proposed to beat the curse of dimensionality \cite{chorin2009implicit,van2010nonlinear}. Secondly, under natural assumptions on the dynamics-observation model, the optimal particle filter can be formulated as a random dynamical system which is very similar to the 3DVAR algorithm. This means that techniques for proving accuracy for the 3DVAR filter in earlier literature \cite{sanz2015long} can be leveraged for the OPF.     
\par
Throughout the article, we make the assumption of conditional Gaussianity for the dynamics-observation model. 
This framework is frequently employed in practice, particularly in geoscience data assimilation problems. Under this assumption, we show that the true posterior, the filtering distribution, satisfies the properties of long-time accuracy 
and of minorization. 
The accuracy result states that, if sufficiently many variables are observed,  
the posterior will concentrate around the true trajectory in the long time 
limit. The minorization result shows that the transition kernel of the
nonlinear Markov process generating the filtering distribution is bounded
below, uniformly in time, by a time-dependent multiple of a fixed 
time-independent probability measure. Minorization may be used
to prove coupling and ergodicity 
in linear Markov processes and this fact will be
exploited when we study particle filters.  
Conditional ergodicity results exploiting coupling are obtained under quite 
general assumptions in \cite{tong2012ergodicity,van2009stability}. 
\par
Having introduced concepts in the context of the filtering distribution
itself, we go on to show that, under the same model-observation assumptions, the OPF exhibits the long-time properties of conditional ergodicity and accuracy for any fixed ensemble size. For the conditional ergodicity result, we show that 
two copies of the particle ensembles, initialized differently, but updated with the same observational data, will converge to each other in the long term limit, in a distributional sense. The accuracy result shows that all ensemble
members in the particle filter will concentrate near the true signal
underlying the data, in the large-time regime. Both the accuracy and
ergodicity results use very similar arguments to those employed for the 
analysis of the filtering distribution itself. 
In recent work, ergodicity has been used to study the long-run asymptotic
behaviour of particle based optimization algorithms \cite{tadic2018asymptotic},
with motivation taken from parameter estimation in partially observed
dynamical systems \cite{poyiadjis2011particle}.
\par
In addition, we also establish large ensemble consistency results for the OPF. Here we employ a technique exposed very clearly in \cite{vanhandel15}, which finds a recursion that is approximately satisfied by the bootstrap particle filter, and leverages this fact to obtain an estimate on the distance between the true posterior and the empirical approximation. We show that the same idea can be applied to not only the OPF, but a very large class of sequential importance sampling procedures. We note that large particle consistency results for particle filters should not be considered practical results for high dimensional data assimilation problems, as in practice particle filters are never implemented in this regime. The consistency results are included here as they are practically informative for low dimensional data assimilation problems and moreover as the results are natural consequences of the random dynamical system formulation that has been adopted for accuracy and ergodicity results. For high dimensional data assimilation problems, it may be more practical to look at covariance consistency, 
as done in \cite{tong2016enkf}. We also note that quite general consistency
results  are proved for optimal particle filters 
in the paper \cite{johansen2008note}, using the setting of auxiliary particle
filters introduced in \cite{pitt1999filtering}; this work was taken
further in \cite{douc2009optimality}.  
\par
As well as obtaining results concerning the stability, accuracy and consistency for the OPF, for which we perform resampling at the end of each assimilation cycle, we also prove the corresponding results for the so called \emph{Gaussianized} OPF. The terminology Gaussianized OPF was introduced in \cite{kelly16b}, but the
idea was introduced two decades ago in \cite{pitt1999filtering} in the
context of the auxiliary particle
filter; see also the paper \cite{liu1995blind}. 
The method differs from the OPF only in the implementation of the resampling. Nevertheless, it was shown numerically in \cite{kelly16b} that the GOPF compares favorably to the OPF, particularly when applied to high dimensional models. The analysis in this article lends theoretical weight to the advantages of the GOPF over the OPF. In particular we find that 
the upper bound on the convergence rate for conditional ergodicity for the GOPF 
has favourable dependence on dimension when compared with the OPF. 
\par
\label{sec:intro}\label{s:intro}
\subsection{Structure of Article and Notation}\label{s:notation}
The remainder of the article is structured as follows. At the end 
of this section we introduce some notation and terminology that will be useful in the sequel. In Section \ref{s:bayes}, we formulate the Bayesian problem of data assimilation, introduce the model-observation assumptions under which we
work, and prove the accuracy and minorization results for the true posterior. In Section \ref{s:pf}, we introduce the bootstrap particle filter, optimal particle filter and Gaussianized optimal particle filter. In Section \ref{s:ergo}, we prove the conditional ergodicity results for the optimal particle filters. In Section \ref{s:acc}, we prove the accuracy results for the optimal particle filters. Finally, in Section \ref{s:consist}, we prove the consistency results for the optimal particle filters.  

Throughout we let $\CX$ denote the finite dimensional Euclidean state space and
and we let $\CY$ denote the finite dimensional Euclidean observation space.
We write $\CM(\CX)$ for the set of probability measures on $\CX$. We denote the Euclidean norm on $\CX$ by $|\cdot|$ and for a symmetric positive definite matrix $A \in L(\CX,\CX)$, we define $|\cdot|_{A} = |A^{-1/2}\cdot |$. 
The notation $P(a|b)$ will denote the density of random variable
$a$, conditioned on known $b$. Transition kernels $q(a,\cdot)$
will denote transitions from point $a$, being measures when the second argument
is a Borel set in $\CX$, and being densities when the second argument 
is an element of $\CX$. 
The kernel $q$ will also be appended with subscript $k$ when emphasizing 
that it is inhomogeneous in time. 
A similar notation will be used for kernels $q_{k}$
on space $\CX^N.$ Superscript $k$ will be used in $q^k$ to denote
$k-$fold composition of the kernels. Finally we define $S^N : \CM(\CX) \to \CM(\CX)$ to be the sampling operator $S^N \mu = \frac{1}{N}\sum_{n=1}^N \delta_{u^{(n)}}$ where $u^{(n)}\sim\mu$ are \iid random variables.

\section{Bayesian Data Assimilation}
\label{sec:bayes}\label{s:bayes}

We describe the set-up which encompasses all the work in this paper,
and then study the minorization condition and accuracy for the true
filtering distribution.

\subsection{Set-Up}
The state model is taken to be a discrete time Markov chain $\{u_k\}_{k\geq0}$ taking values in the state space $\CX$. We assume that the initial condition 
$u_0$ of the chain is distributed according to $\mu_0$, where $\mu_0 \in \CM(\CX)$. The transition kernel for the Markov chain is given by $P(u_{k+1} | u_k)$. For each $k \geq 1$, we make an observation of the Markov chain 
\begin{equ}
\label{eq:obsg}
y_{k+1}  = h(u_{k+1}) + \eta_{k+1}\;,
\end{equ}  
where $h : \CX \to \CY$ maps the state space to the observation space, and 
$\eta_k \sim N(0,\Gamma)$ are centred \iid random variables representing
observational noise. We denote by $Y_k=(y_1,\dots, y_k)$ the accumulated observational data up to time $k$. We are interested in studying and approximating the filtering distribution $\mu_k (\cdot) = \bbP (u_k \in \cdot | Y_k)$ for all $k \geq 1$. We will denote the density of $\mu_k$ by $P(u_k | Y_k)$.  
\par
Although we will not use it for all of the consistency results at the
end of the paper, for the ergodicity and accuracy results
we will always require the following additional assumptions on the 
dynamics-observation model:
\begin{ass}\label{ass:cond_gauss} Let $\psi:\CX \mapsto \CX$ be bounded.
The dynamics-observation model is given by
\begin{subequations}
\label{eq:11}
\begin{align}
u_{k+1}&=\psi(u_{k})+\xi_k, \label{eq:11a}\\
y_{k+1}&=Hu_{k+1}+\eta_{k+1}, \label{eq:11b}
\end{align}
\end{subequations}
where $u_0 \sim \mu_0, \xi_k \sim N(0,\Sigma)$ i.i.d., 
$\eta_k \sim N(0,\Gamma)$ i.i.d. and $u_0, \{\xi_k\}$ and $\{\eta_k\}$ are
independent. We write $\Sigma=\sigma^2 \Sigma_0$ and $\Gamma=\gamma^2 \Gamma_0$ and
require that $\Sigma_0$ and $\Gamma_0$ are strictly positive-definite, and that
$\sigma,\gamma \in (0,\infty)$ so that $r:=\sigma/\gamma \in (0,\infty).$ 
\end{ass}

\par
For most of the results in this article we will be interested in properties 
of the conditional distributions $P(u_k | Y_k)$, and particle approximations
of it, when the observational data $Y_k$ is generated by a fixed realization of
the model. For this reason, we introduce the following notation to emphasize 
that we are considering a fixed realization of the data, generated by a fixed 
trajectory of the underlying dynamical system. 

\begin{ass}\label{ass:dagger}
Fix $\ud_0 \in \CX$ and positive semi-definite matrices $\Sigma_*$ and $
\Gamma_*$ on $\CX$ and $\CY$ respectively.
Let $\{\ud_k\}$ be a realization of the dynamics satisfying
\[
\ud_{k+1}=\psi(\ud_k)+r\gamma \xid_k
\]
where $\ud_0 \in \CX$ is fixed and  $\xid_k \sim N(0,\Sigma_*)$  i.i.d. 
Similarly define $\{\yd_k\}$ by 
\begin{equ}\label{e:y_dagger}
\yd_{k+1}=H\ud_{k+1}+\gamma \etd_{k+1}
\end{equ}
where $\etd_{k+1} \sim N(0,\Gamma_*)$ i.i.d. and $\{\xi_k\}, \{\eta_k\}$
are independent. We will refer to $\{\ud_k\}_{k\geq0}$ as the \emph{true signal} and $\{\yd_k\}_{k\geq1}$ as the given \emph{fixed data}. As above, we use the shorthand $Y_k^\dagger = \{\yd_i\}_{i=1}^k$.
\end{ass}

\begin{rmk}
Note that this data is not necessarily generated from the same statistical
model used to define the filtering distribution both since $r^2\gamma^2\Sigma_*$ 
and $\gamma^2\Gamma_*$
may differ from $\Sigma$ and $\Gamma$, and since the
initial condition is fixed.
However the covariance structures match if we 
define $\sigma:=r\gamma$, 
$\Sigma_*=\Sigma_0$ and $\Gamma_*=\Gamma_0$. When studying accuracy,
we will consider families of data sets and truths parameterized by 
$\gamma \to 0$; in this setting it is natural to
think of $r$ and the noise sequences $\{\xid_k\}_{k\geq0}$ and 
$\{\etd_k\}_{k\geq0}$ as fixed, whilst the true signal and fixed data 
sequences will depend on the value of $\gamma$.
$\qed$
\end{rmk}

The filtering distribution evolves according to the iteration 
\begin{equ}\label{e:bpf_rec2}
\mu_{k+1} = L_{k+1} P \mu_k
\end{equ}
where $P$ and $L_{k+1}$ are maps on the space of measures defined as follows.
The linear map $P: \CM(\CX) \to \CM(\CX)$ is
the Markov semigroup
\[
P \nu (A) = \int_A P(u_{k+1} | u_k) \nu(du_k).
\]
We define the nonlinear likelihood multiplication operator 
$L_{k+1} : \CM(\CX)\to \CM(\CX)$ by  
\begin{equ}
L_{k+1} \nu (A) = \frac{\int_A P(y_{k+1} | u_{k+1}) \nu(du_{k+1})}{\int_\CX P(y_{k+1} | u_{k+1}) \nu(du_{k+1})}
\end{equ}
for each $A \subset \CX$ measurable. 
Equation \eqref{e:bpf_rec2} simply represents application of 
Bayes' formula with prior $P\mu_k$ and likelihood $P(y_{k+1} | u_{k+1}).$ 

Expressed in terms of densities \eqref{e:bpf_rec2} becomes 
\begin{equation}\label{e:bpf_bayes}
\begin{aligned}
P(u_{k+1}|Y_{k+1}) &= P(u_{k+1}|Y_k,y_{k+1})\\
&=\frac{1}{P(y_{k+1}|Y_k)} P(y_{k+1}|u_{k+1},Y_k)P(u_{k+1}|Y_k)\\
&=\int_{\CX} \frac{1}{P(y_{k+1}|Y_k)} P(y_{k+1}|u_{k+1},u_{k},Y_{k})P(u_{k+1}|u_k,Y_k)P(u_k|Y_k)du_k\\
&=\int_{\CX} \frac{1}{P(y_{k+1}|Y_k)} P(y_{k+1}|u_{k+1})P(u_{k+1}|u_k)P(u_k|Y_k)du_k.
\end{aligned}
\end{equation}
Thus we may write
\begin{equation}
\label{eq:markovin}
P(u_{k+1}|Y_{k+1})=\int_{\CX} q_{k+1}(u_k,u_{k+1})P(u_k|Y_k)du_k
\end{equation}
where the transition density $q_{k+1}$ has the form
\begin{equation}
\label{eq:mightneed}
q_{k+1}(u_k,u_{k+1})=\frac{1}{Z}
 \exp\left( -\frac{1}{2} | y_{k+1} - H u_{k+1}  |_\Gamma^2 - \frac{1}{2} | u_{k+1} - \psi(u_k) |_\Sigma^2 \right).
\end{equation}
If we define
\begin{equation*}
q_{k+1}^0(u_k,u_{k+1})=
 \exp\left( -\frac{1}{2} | y_{k+1} - H u_{k+1}  |_\Gamma^2 - \frac{1}{2} | u_{k+1} - \psi(u_k) |_\Sigma^2 \right).
\end{equation*}
then we see that
\begin{equation}
\label{eq:zee}
Z=\int_{\CX}\int_{\CX} q_{k+1}^0(u_k,u_{k+1}) P(u_k|Y_k)du_k du_{k+1}.
\end{equation}
Note the inhomogeneous and nonlinear nature of the Markov
chain reflected in the fact that $Z$ depends on 
$y_{k+1}, u_k$, and on $P(u_k|Y_k).$ 
Despite this dependence, the normalization constant $Z$ can be bounded above
independently of $k$. To see this note that \eqref{eq:zee} gives
\begin{equation*}
Z \le \sqrt{(2\pi)^{d} \det \Sigma}\int_{\CX}\int_{\CX} \frac{1}{\sqrt{(2\pi)^{d} \det \Sigma}}\exp\left( - \frac{1}{2} | u_{k+1} - \psi(u_k) |_\Sigma^2 \right) P(u_k|Y_k)du_k du_{k+1}.
\end{equation*}
Integrating over $u_{k+1}$ first, and then over $u_k$, gives
\begin{equation}
\label{eq:mightneed2}
Z \le \sqrt{(2\pi)^{d} \det \Sigma}.
\end{equation}

In the next two subsections we state two theorems concerning the minorization
conditions and accuracy of 
the filtering distribution itself, followed by a subsection which
elaborates the connection between the optimal particle filter
and the 3DVAR algorithm. 
The remainder of the paper is devoted to establishing analogous
results for various particle filters.
\par
\subsection{Minorization Condition}
\label{ssec:CE}

The result of this subsection is:

\begin{thm}
\label{t:1}\label{thm:bayes_erg}
Consider the filtering distributions $\mu_k$ under 
Assumption \ref{ass:cond_gauss}. 
Assume moreover that the observational data used to define the 
filtering distribution is 
given by $\{\yd_k\}_{k\geq 1}$ from Assumption \ref{ass:dagger}.
Consider the transition  
kernel $q_{k+1}(u_k,\cdot)$ viewed as a random measure, parameterized
by the random observational data. Then
there is a probabiltiy measure $\bbQ \in \CM(\CX)$ 
and a sequence of random constants $\epsilon_k > 0$, defined by
the observational data, such that, for
all Borel sets $A$ in $\CX$, 
\begin{equ}\label{e:bayes_minor}
q_{k+1}(u,A) \geq \epsilon_k \bbQ(A).
\end{equ}
\end{thm}

\begin{proof}
Recall Assumption \ref{ass:cond_gauss} and define 
$\rd_{k,0}=\sigma H\xid_k+ \gamma \etd_{k+1}$ where
$\xid_k \sim N(0,\Sigma_*)$ i.i.d.\,,
$\etd_k \sim N(0,\Gamma_*)$ i.i.d.\,.
Recalling \eqref{eq:mightneed}, \eqref{eq:mightneed2} we obtain
the lower bound 
\begin{align}
\sqrt{(2\pi)^{d} \det \Sigma}\,\, q_{k+1}(u,dv)&= 
\frac{\sqrt{(2\pi)^{d} \det \Sigma}}{Z}
 \exp\left( -\frac{1}{2} | y_{k+1} - H v  |_\Gamma^2 - \frac{1}{2} | v - \psi(u) |_\Sigma^2 \right)dv\\
&\ge  \exp\Bigl(-\frac12|H\psi(\ud_k)+\rd_{k,0}-Hv|^2_{\Gamma}-\frac12
|v-\psi(u)|_{\Sigma}^2\Bigr)dv\\
&\ge \exp\Bigl(-2|H\psi(\ud_k)|^2_{\Gamma}-|\psi(u)|^2_{\Sigma}-2|\rd_{k,0}|^2_{\Gamma}-|Hv|_{\Gamma}^2-|v|_{\Sigma}^2\Bigr)dv\\
&\ge \exp\Bigl(-\lambda^2-2|\rd_{k,0}|^2_{\Gamma}\Bigr)\exp\Bigl(-\frac12|v|_{D}^2\Bigr)dv\\
\end{align}
where
\begin{equ}
\lambda^2=\sup_{u,v} \Bigl(2|H\psi(v)|^2_{\Gamma}+|\psi(u)|^2_{\Sigma}\Bigr)
\end{equ}
and
\[
\frac12 D^{-1}=\Sigma^{-1}+H^*\Gamma^{-1}H\;.
\]
Thus we have a minorization condition of the form
\eqref{e:bayes_minor}
where $\bbQ(\cdot)$ is the Gaussian $N(0,D)$ and
the data-dependent random constants $\epsilon_k$ are given by
$$\epsilon_k=\frac{\sqrt{\det D}}{\sqrt{\det \Sigma}}\exp\Bigl(-\lambda^2-2|\rd_{k,0}|^2_{\Gamma}\Bigr).$$
\end{proof}

\begin{rmk}
For linear Markov processes the existence of a minorization condition
leads in a straightforward way to ergodicity via coupling arguments.
In these arguments a new Markov chain, equivalent in law to the original
one, is used; in this new Markov chain
moves are made according to kernel $\bbQ$ with
probability $\epsilon_k$, being governed by a Bernoulli process.
In the linear case it is posible to fix a realization of the Bernoulli
process and then average with respect to it and obtain a process
equivalent to the original one. This facilitates coupling.
For nonlinear Markov processes fixing the Bernoulli process and then
averaging does not lead to a process equivalent to the original one
and so coupling arguments are more complex; see, for example,
\cite{douc2009forgetting}. The filtering distribution
is governed by a nonlinear Markov process and hence we do not prove
ergodicity. However for the particle filters studied in later sections
we work with linear Markov processes governing the particle ensemble;
the proof of minorization is structurally similar to that in the
preceding theorem and this is why we include the preceding theorem
here. 
\end{rmk}

\subsection{Accuracy}
We now discuss accuracy of the posterior filtering distribution in
the small noise limit $\gamma \ll 1$. The assumptions are somewhat
restrictive, but give a flavour of what can be achieved; more careful
use of ideas from control theory, such as observability, detectability
and stabilizibility may lead to improved results, similar in flavour.
The result uses the 3DVAR filter as an upper bound, and in the next
subsection we will also show that the 3DVAR filter connects very
naturally with the filtering distribution itself, and with
the optimal particle filter in the next section of the paper.

\begin{ass}
\label{ass:contraction}
Let $r = \sigma / \gamma$ and assume that there is $r_c>0$ such that, 
for all $r \in [0,r_c)$, the
function $(I-KH)\psi(\cdot)$, with $K$ defined through \eqref{e:CS}
and \eqref{eq:xtra}, 
is globally Lipschitz on $\CX$ with respect to the norm $\|\cdot\|$
with Lipschitz constant $\alpha=\alpha(r)<1.$
\end{ass}

\begin{thm}
\label{t:acc}
Suppose Assumptions \ref{ass:cond_gauss}, \ref{ass:contraction} 
hold for some $r_c > 0$. Then for all $r \in [0,r_c)$ we have
\[
\limsup_{k \to \infty} \bbE\|u_k-\bbE^{\mu_k}u_k\|^2 \le c\gamma^2\;,
\]
where $\bbE^{\mu_k}$ denotes expectation with respect to measure $\mu_k$ 
defined through \eqref{e:bpf_rec2}
and $\bbE$ denotes expectation over the dynamical model and the observational data.  
\end{thm}

\begin{proof}
This follows similarly to Corollary 4.3 in \cite{sanz}, using the fact
that the mean of the filtering distribution is optimal in the sense that
$$\bbE\|u_k-\bbE^{\mu_k}u_k\|^2 \le \bbE\|u_k-m_k\|^2$$
for any $Y_k-$measurable sequence $\{m_k\}.$
We use for $m_k$ the 3DVAR filter
$$m_{k+1}=(I-KH)\psi(m_k)+Ky_{k+1}.$$
Let $e_k=u_k-m_k$.
Following closely Theorem 4.10 of \cite{LSZ15} we obtain 
$$\bbE \|e_{k+1}\|_{k+1}^2 \le \alpha^2 \bbE \|e_k\|^2+{\mathcal O}(\gamma^2).$$
Application of the Gronwall lemma, plus use of the optimality property, 
gives the required bound.
\end{proof}

\subsection{Connection With 3DVAR}

In the previous subsection we used 3DVAR as a test function to upper
bound the error in the true filtering distribution. Here we further 
develop connections with 3DVAR with an eye on the formulation of
the optimal particle filter as a random dynamical system. 
Cosider the general filtering distribution.
Application of Bayes' formula in the form
\begin{equ}
\label{eq:dda}
q_{k+1} (u_k,du_{k+1})\propto P(y_{k+1} | u_{k+1}) P(u_{k+1} | u_k )du_{k+1}
\end{equ}
gives
\begin{equation*}
q_{k+1} (u_k , du_{k+1}) \propto \exp\left( -\frac{1}{2} | y_{k+1} - h(u_{k+1})  |_\Gamma^2 - \frac{1}{2} | u_{k+1} - \psi(u_k) |_\Sigma^2 \right)du_{k+1},
\end{equation*}
initialized at the measure $\mu_0$. 
Assumption \ref{ass:cond_gauss}, namely the linearity of the observation
operator, introduces a key simplification in this expression: we obtain 
\begin{equ}\label{e:bayes_pk}
q_{k+1} (u_k , du_{k+1}) \propto \exp\left( -\frac{1}{2} | y_{k+1} - H u_{k+1}  |_\Gamma^2 - \frac{1}{2} | u_{k+1} - \psi(u_k) |_\Sigma^2 \right)du_{k+1}
\end{equ}
and a simple completion of the square yields an alternative representation for the transition kernel, namely
\begin{equ}\label{e:bayes_pk2}
q_{k+1} (u_k,du_{k+1}) \propto \exp\left( -\frac{1}{2} |y_{k+1} - H\psi(u_k)|_S^2 -\frac{1}{2} |u_{k+1} - m_{k+1}|_C^2 \right) du_{k+1}\;,
\end{equ}
where 
\begin{equation}\label{e:CS}
\begin{aligned}
C^{-1} &= \Sigma^{-1} + H^* \Gamma^{-1} H \quad,\quad S = H \Sigma H^* + \Gamma \\
m_{k+1} &= C \left( \Sigma^{-1} \psi(u_k) + H^* \Gamma^{-1} y_{k+1}   \right)\;.
\end{aligned}
\end{equation}
The conditional mean $m_{k+1}$ is often given in \emph{Kalman filter} form
\begin{equ}\label{e:m_gain}
m_{k+1} = (I - KH)\psi(u_k) + Ky_{k+1}\;,
\end{equ} 
where $K$ is the \emph{Kalman gain} matrix 
\begin{equ}\label{eq:xtra}
K = \Sigma H^* S^{-1}.
\end{equ}
The expression \eqref{e:bayes_pk} arises from application of Bayes' formula, 
derived above in \eqref{e:bpf_bayes}--\eqref{eq:mightneed}, in the form
\eqref{eq:dda},
whilst \eqref{e:bayes_pk2} follows from a second application of
Bayes' formula to derive the identity
\[
P(y_{k+1} | u_{k+1}) P(u_{k+1} | u_k )du_{k+1} = P(y_{k+1} | u_{k}) P(u_{k+1} | u_k, y_{k+1})du_{k+1}.
\]
We note that
\begin{equation}\label{e:cond_gaus_Ps}
\begin{aligned}
P(y_{k+1} | u_{k}) = Z_{S}^{-1} \exp\left(-\frac{1}{2} |y_{k+1} - H \psi(u_k)|_S^2 \right) \\
P(u_{k+1} | u_k , y_{k+1}) = Z_{C}^{-1} \exp\left(-\frac{1}{2} |u_{k+1} - m_{k+1}|_C^2 \right)\;.
\end{aligned}
\end{equation} 
These formulae are prevalent in the data assimilation literature; in particular 
\eqref{e:m_gain} describes the evolution of the mean state estimate in the 
cycled 3DVAR algorithm, setting $u_{k+1} = m_{k+1}$  \cite{LSZ15}. 
We will make use of the formulae in section \ref{s:particles} 
when describing optimal particle filters as random dynamical systems.

\section{Particle Filters With Resampling}
\label{s:particles}\label{s:pf}
In this section we introduce the bootstrap particle filter, and the two
optimal particle filters, in all three cases with resampling at every step.  
Assumption \ref{ass:cond_gauss} 
ensures that the three particle filters have an elegant interpretation as a 
random dynamical system (RDS) which, in addition, is useful for
our analyses. We thus introduce the filters in this way before 
giving the algorithmic definition which is more commonly found in the 
literature. The bootstrap particle filter will not be the focus of subsequent 
theory, but does serve as an important motivation for the optimal particle 
filters, and in particular for the consistency results in 
Section \ref{s:consist}.
\par
For each of the three particle filters we will make frequent use of a resampling operator, which draws a sample $u_k^{(n)}$ from $\{\uhat_k^{(m)}\}_{m=1}^N$ 
with weights $\{w_k^{(m)}\}_{m=1}^N$ which sum to one. To define this operator, we define the intervals $I_{k}^{(m)} = [\alpha_k^{(m)},\alpha_k^{(m+1)})$ 
where 
\begin{align*}
\alpha_k^{(m+1)} &= \alpha_k^{(m)} + w_k^{(m+1)}, \quad m=0, \cdots N-1\\
\alpha_k^{(0)} &=0.
\end{align*} 
We then set
\begin{equ}\label{e:resampler}
u_k^{(n)} = \sum_{m=1}^N \unit_{I_{k}^{(m)}} (r_k^{(n)}) \uhat_k^{(m)}
\end{equ}
where $r_k^{(n)} \sim U(0,1)$ \iid. Since the weights sum to one, $r_k^{(n)}$ will lie in exactly one of the intervals $I_k^{(i_*)}$ and we will have $u_{k}^{(n)} = \uhat_k^{(i_*)}$. We also notice that
\[
\sum_{m=1}^N \frac{1}{N}\delta_{u_k^{(m)}}=
S^N\sum_{m=1}^N w_k^{(m)}\delta_{\uhat_k^{(m)}}
\]
where $S^N$ is the sampling operator defined previously.

\subsection{The Bootstrap Particle Filter}
The bootstrap particle filter (BPF) approximates the filtering distribution $\mu_k$ with an empirical measure
\begin{equ}\label{e:bpfe}
\bpfe_k^N = \sum_{n=1}^N \frac{1}{N}\delta_{u_k^{(n)}}\;.
\end{equ}
The particle positions $\{u_k^{(n)}\}_{n=1}^N$ are defined as follows. 
\begin{equation}\label{e:bpf_rds}
\begin{aligned}
\uhat_{k+1}^{(n)}&= \psi(u_k^{(n)})+\xi_k^{(n)}\quad,\quad\text{$\xi_k^{(n)} \sim N(0,\Sigma)$ \iid},\\
u_{k+1}^{(n)}&=\sum_{m=1}^{N} {\bbI}_{I_{k+1}^{(m)}}(r_{k+1}^{(n)})\uhat_{k+1}^{(m)}\;,
\end{aligned}
\end{equation}
where  the second equation uses the resampling operator defined 
in \eqref{e:resampler} with weights computed according to 
\begin{equ}\label{e:pf_weight_cd}
w_{k+1}^{(n),*} = \exp(-\frac{1}{2}|y_{k+1} - H\uhat_{k+1}^{(n)}|_\Gamma^2) \quad ,\quad w_{k+1}^{(n)} = \frac{w_{k+1}^{(n),*}}{\sum_{j=1}^N w_{k+1}^{(j),*}} \;.
\end{equ}
Thus, for each particle in the RDS, we propagate them forward using the dynamical model and then re-sample from the weighted particles to account for the observation likelihood.  
\par
Recall Bayes formula \eqref{e:bpf_bayes}.
The bootstrap particle filter approximates the posterior via a sequential application of importance sampling, using $P(u_{k+1}|Y_k) = \int P(u_{k+1} | u_k) P(u_k | Y_k) du_k$ as the proposal and re-weighting according to the likelihood $P(y_{k+1}|u_{k+1})$. Thus the method is typically described by the following algorithm for updating the particle positions. The particles are initialized with $u_0^{(n)}\sim\mu_0$ and then updated iteratively as follows: 
\begin{enumerate} 
\item Draw $\uhat_{k+1}^{(n)} \sim P(u_{k+1} | u_k^{(n)})$.
\item Define the weights $w_{k+1}^{(n)}$ for $n=1,\dots,N$ by
\begin{equ}
w_{k+1}^{(n),*} = P(y_{k+1} | \uhat_{k+1}^{(n)}) \quad ,\quad w_{k+1}^{(n)} = \frac{w_{k+1}^{(n),*}}{\sum_{m=1}^N w_{k+1}^{(m),*}} \;.
\end{equ}
\item  Draw $u_{k+1}^{(n)}$ from $\{\uhat_{k+1}^{(n)}\}_{n=1}^N$ with weights $\{w_{k+1}^{(n)}\}_{n=1}^N$. 
\end{enumerate}
Under Assumption \ref{ass:cond_gauss}, it is clear that the sampling and re-weighting procedures are consistent with \eqref{e:bpf_rds}. Note that the normalization factor $P(y_{k+1} | Y_k)$ is not required in the algorithm and is instead approximated via the normalization procedure in the second step. 
\par
In addition to $\bpfe_k^N$ it is also useful to define the related measure 
\begin{equ}\label{e:bpfw}
\bpfw^N_k = \sum_{n=1}^N w_k^{(n)} \delta_{\uhat_k^{(n)}}\;,
\end{equ}
 with $\bpfw_0^N = \mu_0$, which is related to the bootstrap particle filter by $\bpfe_k^N = S^N \bpfw_k^N$. As we shall see in Section \ref{s:consist}, the advantage of $\bpfw_k^N$ is that it has a recursive definition which allows for elegant proofs of consistency results \cite{vanhandel15}.  
  
\subsection{Optimal Particle Filter}

The optimal particle filter with resampling can also be formulated as a RDS. 
We once again approximate the filtering distribution $\mu_k$ with an empirical distribution
\begin{equ}\label{e:opfe}
\opfe_k^N = \sum_{n=1}^N \frac{1}{N}\delta_{u_k^{(n)}}\;.
\end{equ}
{Under  Assumption \ref{ass:cond_gauss} the particles in this approximation
are defined as follows.}
The particle positions are initialized with $u_0^{(n)} \sim \mu_0$.
Given a collection of particles $u_k^{(n)}$ the particles are evolved
according to the RDS update step 
\begin{equation}\label{e:opf_rds}
\begin{aligned}
\uhat_{k+1}^{(n)}&=(I-KH)\psi(u_k^{(n)})
+Ky_{k+1}+\zeta_k^{(n)}  \quad,\quad \text{$\zeta_k^{(n)} \sim N(0,C) $ \iid}\\
u_{k+1}^{(n)}&=\sum_{m=1}^{N} {\bbI}_{I_{k+1}^{(m)}}(r_{k+1}^{(n)})\uhat_{k+1}^{(m)}\;.
\end{aligned}
\end{equation}
Here $C,S,K$ are defined in \eqref{e:CS}, \eqref{e:m_gain} and as with the BPF, the second equation uses the resampling operator defined in \eqref{e:resampler} but now using weights computed by
\begin{equ}\label{e:opf_weight_cd}
w_{k+1}^{(n),*} = \exp(-\frac{1}{2}|y_{k+1} - H\psi(u_{k}^{(n)})|_S^2) \quad ,\quad w_{k+1}^{(n)} = \frac{w_{k+1}^{(n),*}}{\sum_{m=1}^N w_{k+1}^{(m),*}} \;.
\end{equ}
In light of the formulae given in \eqref{e:cond_gaus_Ps}, which are
derived under  Assumption \ref{ass:cond_gauss}, we see that the optimal particle filter is updating the particle positions by sampling from $P(u_{k+1} | u_k^{(n)},y_{k+1})$ and then re-sampling to account for the likelihood factor $P(y_{k+1} | u_k^{(n)})$. In particular, without necessarily making  Assumption \ref{ass:cond_gauss}, the optimal particle filter is a sequential importance sampling scheme applied to the following decomposition of the filtering distribution
\begin{equation}\label{e:opf_bayes}
\begin{aligned}
P(u_{k+1} | Y_{k+1}) &= \int_{\CX} P(u_{k+1} , u_k | Y_{k+1}  ) du_k \\
& = \int_{\CX} P(u_{k+1} | u_k , y_{k+1}) P(u_k | Y_{k+1}) du_k\\
& = \int_{\CX} \frac{P(y_{k+1} | u_k)}{P(y_{k+1}| Y_k)}P(u_{k+1} | u_k , y_{k+1}) P(u_k | Y_{k}) du_k\;.
\end{aligned}
\end{equation}
In the algorithmic setting, the filter is initialized with $u_0^{(n)} \sim \mu_0$, then for $k \geq 0$
\begin{enumerate}
\item Draw $\uhat_{k+1}^{(n)}$ from $P(u_{k+1}| u_k^{(n)}, y_{k+1})$
\item Define the weights $w_{k+1}^{(n)}$ for $n=1,\dots,N$ by
\begin{equ}
w_{k+1}^{(n),*} = P(y_{k+1} | \uhat_{k}^{(n)}) \quad ,\quad w_{k+1}^{(n)} = \frac{w_{k+1}^{(n),*}}{\sum_{m=1}^N w_{k+1}^{(m),*}} \;.
\end{equ}
\item Draw $u_{k+1}^{(n)}$ from $\{\uhat_{k+1}^{(m)}\}_{m=1}^N$ with weights $\{w_{k+1}^{(m)}\}_{m=1}^N$.
\end{enumerate}
It is important to note that, although the OPF is well defined in this general setting for any choice of dynamics-obsevation model, 
it is only implementable under stringent assumptions on the
forward and observation model, such as those given in
Assumption \ref{ass:cond_gauss};
under this assumption the steps 1 and 2 may be implemented using 
the formulae given in \eqref{e:cond_gaus_Ps} and exploited in the
derivation of \eqref{e:opf_rds}. We emphasize that
models satisfying Assumption \ref{ass:cond_gauss} do arise frequently in
practice.
\par
As with the BPF, it is beneficial to consider the related particle filter given by 
\begin{equ}\label{e:opfw}
\opfw_k^N = \sum_{n=1}^N w_{k}^{(n)} \delta_{\uhat_k^{(n)}}
\end{equ}
for $k \geq 1$ and with $\opfw_0^N = \mu_0$. Similarly to the bootstrap
particle filter  we have that $\opfe_k^N = S^N\opfw_k^N$. 

\subsection{Gaussianized Optimal Particle Filter}

In \cite{kelly16b}, an alternative implementation of the OPF is investigated and found to have superior performance on a range of test problems, particularly with respect to the curse of dimensionality. We refer to this filter as the Gaussianized optimal particle filter (GOPF), but note that it was first derived
in \cite{pitt1999filtering}.
Once again, we approximate the filtering distribution with an empirical measure
\begin{equ}\label{e:gopf}
\gopf_k^N = \sum_{n=1}^N \frac{1}{N}\delta_{v_k^{(n)}}\;.
\end{equ}
As in the previous subsection, we first describe the filter under
Assumption \ref{ass:cond_gauss}.
The filter is initialized with $v_0^{(n)} \sim \mu_0$, with subsequent
iterates generated by the RDS
\begin{equation}\label{e:gopf_rds1}
\begin{aligned}
\vtilde_{k}^{(n)}&=\sum_{m=1}^{N} {\bbI}_{I_{k+1}^{(m)}}(r_{k+1}^{(n)})v_{k}^{(m)}\;,\\
v_{k+1}^{(n)}&=(I-KH)\psi(\vtilde_k^{(n)})
+Ky_{k+1}+\zeta_k^{(n)} ,\quad \text{$\zeta_k^{(n)} \sim N(0,C) $ \iid}
\end{aligned}
\end{equation}
and the weights appearing in the resampling operator are given by
\begin{equ}\label{e:gopf_weight_cd}
w_{k+1}^{(n),*} = \exp(-\frac{1}{2}|y_{k+1} - H\psi(v_{k}^{(n)})|_S^2),\quad w_{k+1}^{(n)} = \frac{w_{k+1}^{(n),*}}{\sum_{m=1}^N w_{k+1}^{(m),*}} \;.
\end{equ}
Thus, the update procedure for GOPF is \emph{weight-resample-propagate}, as opposed to \emph{propagate-weight-resample} for the OPF. Hence the only difference between the OPF and GOPF is the ordering of the the resampling and propagation
steps. 
\par
In our analysis it is sometimes useful to consider the equivalent RDS
\begin{equation}
\begin{aligned}
\label{e:gopf_rds2}
\vhat_{k+1}^{(m,n)}&=(I-KH)\psi(v_k^{(m)})
+Ky_{k+1}+\zeta_k^{(m,n)} \quad \text{$\zeta_k^{(m,n)} \sim N(0,C)$ \iid}\\
v_{k+1}^{(n)}&=\sum_{m=1}^{N} {\bbI}_{I_{k+1}^{(m)}}(r_{k+1}^{(n)})\vhat_{k+1}^{(m,n)}\;.
\end{aligned}
\end{equation} 
The sequences $v_k^{(n)}$ defined in \eqref{e:gopf_rds1} and \eqref{e:gopf_rds2} agree because for every $n$ there is exactly one $m=m^*(n)$ such that $\vhat_{k+1}^{(m^*(n),n)}$
survives the resampling step.
Writing the algorithm this way allows certain parts of our subsequent
analysis to be performed very similarly for both the OPF and GOPF;
it is not a formulation to be implemented in practice.
\par
For a general dynamics-observation model, the GOPF is described by the 
following algorithm:
\begin{enumerate}
\item Define the weights $w_{k+1}^{(n)}$ for $n=1,\dots,N$ by
\begin{equ}
w_{k+1}^{(n),*} = P(y_{k+1} | v_{k}^{(n)}) \quad ,\quad w_{k+1}^{(n)} = \frac{w_{k+1}^{(n),*}}{\sum_{m=1}^N w_{k+1}^{(m),*}} \;.
\end{equ}
\item Draw $\vtilde_{k}^{(n)}$ from $\{v_{k}^{(m)}\}_{m=1}^N$ with weights $\{w_{k+1}^{(m)}\}_{m=1}^N$.
\item Draw $v_{k+1}^{(n)}$ from $P(u_{k+1}| \vtilde_k^{(n)}, y_{k+1})$\;.
\end{enumerate}   
Unlike for the previous filters, there is no need to define an associated `hatted' measure, as the GOPF can be shown to satisfy a very natural recursion. This will be discussed in Section \ref{s:consist}. 

\section{Ergodicity for Optimal Particle Filters}\label{s:ergo}

In this section we study the conditional ergodicity of the two optimal
particle filters. The proofs are structurally very similar to 
one another and so we give details only in one case.
The ergodicity results require a metric on probability measures to quantify 
convergence of differently initialized posteriors in the long time limit. 
To this end, we define the total variation metric on $\CM(\CX)$ by
\begin{equ}
d_{TV} (\mu,\nu) = \frac12\sup_{|h|\leq 1} |\mu(h) - \nu(h)|
\end{equ}
where the supremum is taken over all bounded functions $h : \CX \to \reals$ with $|h| \leq 1$, and where we define $\mu(h) := \int_{\CX} h(x) \mu(dx)$ 
for any probability measure $\mu \in \CM(\CX)$ and any 
real-valued test function $h$ bounded by $1$ on $\CX$.
This definition is then readily extended to probability
measures on $\CX^N.$

\subsection{Optimal Particle Filter}

Before stating the conditional ergodicity result, we first need some notation. Define $u_k = (u_k^{(1)},\dots,u_k^{(N)})$ to be particle positions defined by the RDS \eqref{e:opf_rds} with $\mu_0 = \delta_{z_0}$ and similarly $u_k' = 
(u_k^{(1)'},\dots, u_k^{(N)'})$ with $\mu_0 = \delta_{z_0'}$. Then $u_k$ is a Markov chain taking values on $\CX^N$, whose inhomoegeous Markov kernel we denote by $q_k(z,\cdot)$. The law of $u_k$ is given by $q^k(z_0,\cdot)$, defined recursively 
by composing the $q_k$; and similarly the law of $u_k'$ is given by $q^k(z_0',\cdot)$. The conditional ergodicity result states that if the two filters $u_k, u_k'$ are driven by the same observational data, then the law of $u_k$ will converge to the law of $u_k'$ exponentially as $k\to\infty$.

\begin{rmk}We abuse notation in this subsection 
by using $u_k \in \CX^N$ to denote the $N$
particles comprising the optimal particle filter; this differs from
the notation $u_k \in \CX$ used in the remainder of the paper to denote
the underlying dynamical model. Similarly $q_k$ is here a one-step
linear Markov kernel on $\CX^N$ whereas, previously, it denoted
a one-step nonlinear Markov kernel on $\CX.$ 
We note also the important distinction between
$q_k$ and $q^k$: in the former case $q_k$ is a one-step transition kernel,
inhomogeneous and depending on $k$; in the latter case $q^k$ is kernel
found by composing over $k$ steps. $\qed$
\end{rmk}

\begin{thm}
\label{thm:opf_erg}
Suppose that  Assumptions \ref{ass:cond_gauss} hold.
Consider the OPF particles $u_k, u_k'$ defined above. Assume moreover that the observational data used to define each filter is the same,
and given by $\{\yd_k\}_{k\geq 1}$ from Assumption \ref{ass:dagger}. Then there exists $\mfz_N \in (0,1)$ such that, almost surely with respect
to the randomness generating $\{\yd_k\}_{k\geq 1}$,
\begin{equ}\label{e:opf_erg}
\limsup_{k \to \infty} \Bigr(\dtv\bigl(q^k(z_0,\cdot),q^k(z_0',\cdot)\bigr)\Bigr)^{1/k} \le \mfz_N \;.
\end{equ}
\end{thm}

\begin{proof}

\emph{Step A}: 
Notice that
\begin{equation}
\label{eq:markovin2}
q^{k+1}(z_0,\cdot)=\int_{\CX} q_{k+1}(u_k,\cdot)q^{k}(z_0,du_k)
\end{equation}
where the transition kernel $q_{k+1}(u_k,\cdot)$ is 
here viewed as being a measure. 


\par
The heart of the argument is Step B, below, in which we 
prove a minorization condition for the transition kernel $q_{k+1}$,
as we did for the filtering distribution itself in subsection \ref{ssec:CE}.
That is, we seek a measure $\bbQ \in \CM(\CX^N)$ and a sequence of constants $\epsilon_k > 0$ satisfying 
\begin{equ}\label{e:bayes_minor_z}
q_{k+1}(u,A) \geq \epsilon_k \bbQ(A)
\end{equ}   
for all $u \in \CX^N$ and all measurable sets $A \subset \CX^N$. 
Given a minorization condition, we obtain the result via the following 
standard coupling argument. The minorization condition allows us to define a new Markov kernel 
\begin{equ}
\tp_{k+1}(x,A)=(1-\epsilon_k)^{-1}\Bigl(q_{k+1}(x,A)-\epsilon_k \bbQ(A)\Bigr)\;.
\end{equ}
The Markov chain in which transitions occur with probability $1-\epsilon_k$
according to $\tp_{k+1}(x,\cdot)$ and with probability $\epsilon_k$
according to $\bbQ(\cdot)$ is equivalent in law to the Markov chain
$\{u_k\}$, and similarly for $\{u_k'\}$. Furthermore we
may couple the two Markov chains by using the same random variables
to select whether moves are made according to $\tp_{k+1}(x,\cdot)$ or $\bbQ(\cdot)$. 
We now complete the coupling argument, assuming the minorization condition
holds.
We compare expectations of the two Markov chains, using the equivalent
in law formulation above, and coupled through the random moves according
to $\bbQ(\cdot)$ which occur at each step with probability $\eps_k.$
Let $A_k$ be the event that the state
independent Markov kernel $\bbQ(\cdot)$ is not picked at all times 
$j=0, \cdots, k-1$. Then we have 
\begin{align*}
\dtv\bigl(q^k(z_0,\cdot),q^k(z_0',\cdot)\bigr)&=\frac12 \sup_{|f|_{\infty} \le 1}|\bbE\bigl(f(u_k)-f(u_k')\bigr)|\\
&=\frac12 \sup_{|f|_{\infty} \le 1}\Bigl|\bbE\Bigl(\bigl(f(u_k)-f(u_k')\bigr)\bbI_{A_k} +\bigl(f(u_k)-f(u_k')\bigr)\bbI_{A_k^c}\Bigr)\Bigr|\;.
\end{align*}
Note that for this coupling the second term vanishes, as in the event $A_k^c$, the two chains Markov kernels $q^k(z_0,\cdot)$ and $q^k(z_0',\cdot)$ will have become
identical to measure $\bbQ$ at, or before, step $k$. Once that happens, they remain identical for all future steps. It follows that 
\begin{align*}
\dtv\bigl(q^k(z_0,\cdot),q^k(z_0',\cdot) \bigr) & \le \bbE (\bbI_{A_k})
=\bbP(A_k)
=\Pi_{j=1}^k(1-\epsilon_j)\;.
\end{align*}
To obtain the result \eqref{e:opf_erg}, we need to understand the limiting behaviour of the constants $\epsilon_j$ appearing in the minorization condition \eqref{e:bayes_minor_z}. Hence we turn our attention toward obtaining the minorization condition. 

\emph{Step B}:
Before deriving the minorization, 
we introduce some preliminaries. Using the fact that 
\[
\yd_{k+1}=H\psi(\ud_{k})+\gamma(rH\xid_k+ \etd_{k+1})
\] 
 and defining
\begin{align*}
a_k &= \left( (I-KH)\psi(u_k^{(n)}) + K H \psi(\ud_k) \right)^N_{n=1}\quad \zeta_k=\left(\zeta_k^{(n)}\right)_{n=1}^N \\ \rd_{k,0} &=\gamma (rH\xid_k+ \etd_{k+1}),
\quad \rd_k=\left( \gamma K(rH\xid_k+ \etd_{k+1}) \right)^N_{n=1}
\end{align*}
we see that
\[
\uhat_{k+1}=a_k+\rd_k+\zeta_k\;.
\]
The next element of the sequence, $u_{k+1}$, is then defined by the
second identity in \eqref{e:opf_rds}.
We are interested in the conditional ergodicity of $\{u_k\}_{k=1}^{\infty}$
with the sequence $\{\rd_k\}_{k=1}^{\infty}$ fixed. By Assumption \ref{ass:cond_gauss}, $a_k$ is bounded uniformly in $k$. We define the covariance operator $\mfC \in L(\CX^N,\CX^N)$ to be a block diagonal covariance with each diagonal entry equal
to $C$ and then 
\[
R=\sup_{(u,v)} \Bigl(|(I-KH)\psi(u)+KH \psi(v)|_C^2\Bigr)\;,
\]
which is finite by Assumption \ref{ass:cond_gauss}. 
\par
Now, let $E_0$ be the event that, upon resampling, every particle survives
the resampling. There are $N!$ such permutations. We will do the calculation in the case of a trivial permutation, that is, where each particle is mapped to itself under the resampling.
However the bounds which follow work for any permutation because
we do not use any information about location of the mean of the
particle proposals; we simply
use bounds on the drift $\psi.$ If each particle is mapped to itself, 
then $u^{(n)}_{k+1} = \uhat^{(n)}_{k+1}$ for all $n=1,\dots,N$. It follows
that
\begin{equs}
q_{k+1} (u, A) &= \bbP(u_{k+1} \in A | u_k = u)\\ &\geq \bbP(u_{k+1} \in A | u_k = u , E_0) \bbP(E_0)\\ &= \bbP(\uhat_{k+1} \in A | u_k = u ) \bbP(E_0)\;.
\end{equs}

We now note that
\begin{align*}
\bbP (\uhat_{k+1} \in A | u_k = u )  &=\frac{1}{\sqrt{(2\pi)^{dN} \det \mfC}}\int_A \exp
\Bigl(-\frac12|x-a_k-\rd_k|_{\mfC}^2\Bigr) dx\\
& \ge \frac{\exp\Bigl(-|a_k+\rd_k|^2_{\mfC}\Bigr)}{\sqrt{(2\pi)^{dN} \det \mfC}}
\int_A \exp \Bigl(-|x|_{\mfC}^2\Bigr) dx\\
& \ge 2^{-dN/2}\exp(-2|a_k|_{\mfC}^2)\exp(-2|\rd_k|_{\mfC}^2) \bbQ_{\mfC}(A)\\
& \ge 2^{-dN/2}\exp(-2NR^2)\exp(-2|\rd_k|_{\mfC}^2) \bbQ_{\mfC}(A)
\end{align*}
where $\bbQ_{\mfC}(A)$ is the Gaussian measure $N(0,\frac12\mfC).$
Thus we have shown that
\begin{equation}
\label{eq:co}
\bbP(\uhat_{k+1} \in A | u_k = u ) \ge \delta_k\bbQ_{\mfC}(A)
\end{equation}
where
\[
\delta_k=2^{-dN/2}\exp(-2NR^2)\exp(-2|\rd_k|_{\mfC}^2)\;.
\]
Moreover, we have that 
\[
\bbP(  E_0 ) = N! \Pi_{n=1}^N w_{k+1}^{(n)}.
\]
Note that we have the bound 
$w_{k+1}^{(n)} \geq w_{k+1}^{(n),*}/N$ for each $n=1,\dots,N$ because
each $w_{k+1}^{(m),*}$ is bounded by $1$. But we have
\begin{align*}
w_{k+1}^{(n),*}&=\exp\Bigl(-\frac12|y_{k+1}-H\psi(u_k^{(n)})|_{S}^2\Bigr)\\
&=\exp\Bigl(-\frac12|H\psi(\ud_k)-H\psi(u_k^{(n)})+\rd_{k,0}|_{S}^2\Bigr)\\
& \ge \exp\Bigl(-r^2-|\rd_{k,0}|_{S}^2\Bigr)
\end{align*}
where
\[
r^2= \sup_{u,v}|H\psi(u)-H\psi(v)|_S^2
\]
which is finite by Assumption \ref{ass:cond_gauss}. 
From this we see that 
\[
\bbP(E_0) \ge N! \frac{1}{N^N} \exp\bigl(-Nr^2-N|\rd_{k,0}|_{S}^2\bigr).
\]
Thus we obtain the minorization conditon \eqref{e:bayes_minor_z} where 
\[
\epsilon_k= N! \frac{1}{N^N} \exp\bigl(-Nr^2-N|\rd_{k,0}|_{S}^2\bigr)\delta_k\quad,\quad \bbQ = \bbQ_\mfC\;.
\]

\emph{Step C}: By the argument in Step A we have that
\begin{equ}
\dtv\bigl(q^k(z_0,\cdot),q^k(z_0',\cdot)\bigr)^{1/k} \leq \mfz_k
\end{equ}
where $\mfz_k=\Bigl(\Pi_{j=1}^k(1-\epsilon_j)\Bigr)^{1/k}$. Since the $\epsilon_k$ are \iid and integrable, by the law of large numbers, almost surely with respect to the randomness generating the true signal and the data, we have
\begin{equ}
\ln \mfz_k = \frac{1}{k}\sum_{j=1}^k \ln(1-\epsilon_j) \to \bbE  \ln(1-\epsilon_1)=-\bbE \sum_{n=1}^\infty
\frac{1}{n}\epsilon_1^n\;.
\end{equ}
But $\epsilon_1\le \exp(-2|\rd_{1,0}|_\Gamma^2).$ 
Since $\rd_{1,0}$ is Gaussian it follows that the $n^{th}$ moment 
of $\epsilon_1$ is bounded above by ${\mathcal O}(n^{-\frac12})$ so that
the limit of $\ln \mfz_k$ is negative and finite; the result follows.

\end{proof}

\subsection{Gaussianized Optimal Filter}
As in the last section, we define $v_k = (v_k^{(1)},\dots,v_k^{(N)})$ and similarly for $v_k'$ using the RDS but now for the GOPF \eqref{e:gopf_rds1} (or alternatively \eqref{e:gopf_rds2}) with distinct initializations $\mu_0 = \delta_{z_0}$ and $\mu_0 = \delta_{z_0'}$. Similarly to Theorem \ref{thm:opf_erg}, 
we let $q^{k}(z_0,\cdot)$ denote the law of $v_k$. 

\begin{thm}
\label{thm:gopf_erg}
Suppose that  Assumptions \ref{ass:cond_gauss} hold.
Consider the GOPF particles $v_k, v_k'$ defined above. Assume moreover that the observational data used to define each filter is the same,
and given by $\{\yd_k\}_{k\geq 1}$ from Assumption \ref{ass:dagger}. Then there exists $\mfz_N \in (0,1)$ such that, almost surely with respect
to the randomness generating $\{\yd_k\}_{k\geq 1}$,
\begin{equ}\label{e:gopf_erg}
\limsup_{k \to \infty} \Bigr(\dtv\bigl(q^k(z_0,\cdot),q^k(z_0',\cdot)\bigr)\Bigr)^{1/k} \le \mfz_N \;.
\end{equ}
\end{thm}

\begin{proof}
The proof follows similarly to that of Theorem \ref{thm:opf_erg}, in particular it suffices to obtain a minorization condition for $q_{k+1}(v,\cdot)$. We will use the RDS representation \eqref{e:gopf_rds2}, which we now recall 
\begin{equation}
\begin{aligned}
\vhat_{k+1}^{(m,n)}&=(I-KH)\psi(v_k^{(m)})
+Ky_{k+1}+\zeta_k^{(m,n)} \quad \text{$\zeta_k^{(m,n)} \sim N(0,C)$ \iid}\\
v_{k+1}^{(n)}&=\sum_{m=1}^{N} {\bbI}_{I_{k+1}^{(m)}}(r_{k+1}^{(n)})\vhat_{k+1}^{(m,n)}\;.
\end{aligned}
\end{equation} 
In this formulation, note that for each $n$ there is one and only one $m = m^*(n)$ such that 
${\bbI}_{I_{k}^{(m)}}(r_{k}^{(n)})=1$. We see that 
\[
v_k :=  (v_k^{(n)})_{n=1}^N=(\vhat_{k}^{(m^*(n),n)})_{n=1}^N\;.
\]
Using the fact that 
\[
\yd_{k+1}=H\psi(\ud_{k})+\gamma(rH\xid_k+ \etd_{k+1})
\]
and defining
\begin{equs}
a_k &= \left( (I-KH)\psi(v_k^{(m^*(n))}) + K H \psi(\ud_k) \right)^N_{n=1}\quad \zeta_k=\left(\zeta_k^{(m^*(n),n)}\right)_{n=1}^N  \\ \rd_{k,0}&=\gamma (rH\xid_k+ \etd_{k+1}),
\quad \rd_k=\left( \gamma K(rH\xid_k+ \etd_{k+1}) \right)^N_{n=1}
\end{equs}
we see that
\[
v_{k+1}=a_k+\rd_{k}+\zeta_{k}\;.
\]
Now notice that
\begin{equs}
q_{k+1}(v,A) = \bbP (v_{k+1} \in A | v_k = v) &= \bbP\left( (\vhat_{k+1}^{(m^*(n),n)})_{n=1}^N \in A | v_k = v \right) \\
& = \frac{1}{\sqrt{(2\pi)^{dN} \det \mfC}}\int_A \exp
\Bigl(-\frac12|x-a_{k}-\rd_{k}|_{\mfC}^2\Bigr) dx\\
& \geq \frac{\exp\Bigl(-|a_k+\rd_k|^2_{\mfC}\Bigr)}{\sqrt{(2\pi)^d \det \mfC}}
\int_A \exp \Bigl(-|x|_{\mfC}^2\Bigr) dx\\
& \geq 2^{-dN/2}\exp(-2|a_k|_{\mfC}^2)\exp(-2|\rd_k|_{\mfC}^2) \bbQ_{\mfC}(A)\\
& \geq 2^{-dN/2}\exp(-2NR^2)\exp(-2|\rd_k|_{\mfC}^2) \bbQ_{\mfC}(A)
\end{equs}
where $\bbQ_{\mfC}$ is the Gaussian measure $N(0,\frac12\mfC).$
Thus we have shown that
\begin{equation}
\label{eq:Gco}
q_{k+1}(v,A) \geq \delta_k \bbQ_{\mfC}(A)
\end{equation}
where
\[
\delta_k=2^{-dN/2}\exp(-2NR^2)\exp(-2|\rd_k|_{\mfC}^2)\;.
\]
The remainder of the proof (step C) follows identically 
to Theorem \ref{thm:opf_erg}. 
\end{proof}

\begin{rmk}
We can compare our (upper bounds on the) rates of convergence for the 
two optimal filters, using the minorization constants. 
For the OPF we have
\[
\epsilon_1 = N! \frac{1}{N^N} \exp\bigl(-Nr^2-N|\rd_{1,0}|_{S}^2\bigr)\delta_1\;,
\]
where 
\[
\delta_1=2^{-dN/2}\exp(-2NR^2)\exp(-2|\rd_1|_{\mfC}^2)\;;
\]
for the GOPF we simply have $\epsilon_1 = \delta_1$. The extra $N$ dependence in the OPF clearly leads to a slower (upper bound on the) rate of convergence for the OPF. Thus, 
by this simple argument, we obtain a better convergence rate for the GOPF than for the 
OPF. This suggests that the GOPF may have a better rate of convergence for fixed 
ensemble sizes; further analysis or experimental  study of this point would be of interest. 
$\qed$
\end{rmk}

\section{Accuracy for Optimal Particle Filters}\label{s:acc}

In this section we study the accuracy of the optimal particle filters, in the
small noise limit $\gamma \to 0$. The expectation appearing in the theorem statements
is with respect to the noise generating the data, and with respect to the randomness
within the particle filter itself. Note that this situation differs from that
in the accuracy result for the filter itself which uses data generated by the
statistical model. Assumption \ref{ass:dagger} relaxes this assumption.

\subsection{Optimal Particle Filter}


\begin{thm}
\label{thm:opf_acc}
Suppose that  Assumptions \ref{ass:cond_gauss}, \ref{ass:contraction} hold
and consider the OPF with particles $\{u_k^{(n)}\}_{n=1}^N$ defined by \eqref{e:opf_rds} with data $\{\yd_k\}$
given by Assumption \ref{ass:dagger}.
It follows that there is constant $c$ such that
$$\limsup_{k \to \infty} \bbE \Bigl(\max_n\|u_{k}^{(n)}-\ud_k\|^2\Bigr) \le c\gamma^2.$$
\end{thm}

\begin{proof}
First recall the notation $\Sigma = \sigma \Sigma_0$, $\Gamma = \gamma \Gamma_0$ and $r = \sigma / \gamma$. Now define
\begin{align*}
S_0&=r^2 H\Sigma_0 H^*+\Gamma_0,\\
C_0&=r^2(I-KH)\Sigma_0
\end{align*}
and note that
\[
S=\gamma^2 S_0, C=\gamma^2 C_0, K=r^2 \Sigma_0 H^* S_0^{-1}\;.
\]
We will use the RDS representation
\begin{equation}\label{e:opf_rds_acc}
\begin{aligned}
\uhat_{k+1}^{(n)}&=(I-KH)\psi(u_k^{(n)})
+K\yd_{k+1}+\gamma \zeta_{0,k}^{(n)}  \quad,\quad \text{$\zeta_k^{(n)} \sim N(0,C) $ \iid}\\
u_{k+1}^{(n)}&=\sum_{m=1}^{N} {\bbI}_{I_{k+1}^{(m)}}(r_{k+1}^{(n)})\uhat_{k+1}^{(n)}\;.
\end{aligned}
\end{equation}
where $\zeta_{0,k}^{(n)} \sim N(0,C_0)$ \iid. Hence we have
\begin{equs}
\ud_{k+1}&=(I-KH)\psi(\ud_{k})+KH\psi(\ud_{k})+r\gamma\xid_k,\\
\uhat_{k+1}^{(n)}&=(I-KH)\psi(u_{k}^{(n)})+K(H\psi(\ud_{k})+\gamma \etd_{k+1})+\gamma\ztd_{0,k}.
\end{equs}
where $\ztd_{0,k} \sim N(0,C_0)$ i.i.d. 
Subtracting, 
we obtain
\begin{equation}
\label{eq:ee}
\uhat_{k+1}^{(n)}-\ud_{k+1}=(I-KH)\bigl(\psi(u_{k}^{(n)})-\psi(\ud_{k})\bigr)+
\gamma \iota_k^{(n)}
\end{equation}
where $\iota_k^{(n)}:= (K\etd_{k+1}+\ztd_{0,k}-r\xid_k)$. Moreover we have the identity
\begin{equation}
\label{eq:np2}
\ud_{k+1}=\sum_{m=1}^{N} {\bbI}_{I_{k+1}^{(m)}}(r_{k+1}^{(n)})\ud_{k+1}\;.
\end{equation} 
Thus, defining
\[
e_k^{(n)}=u_k^{(n)}-\ud_k, \quad \he_k^{(n)}=\uhat_k^{(n)}-\ud_k
\]
we have from \eqref{e:opf_rds_acc} and \eqref{eq:np2}
\[
e_{k+1}^{(n)}=\sum_{m=1}^{N} {\bbI}_{I_{k+1}^{(m)}}(r_{k+1}^{(n)})\he_{k+1}^{(m)}\;.
\]
Thus
\[
\max_n \|e_{k+1}^{(n)}\|^2 \le \max_m \|\he_{k+1}^{(m)}\|^2\;,
\]
where the norm is the one in which we have a contraction. Using \eqref{eq:ee}, the Lipschitz property of $(I-KH)\psi(\cdot)$, taking expectations and using independence, yields
\[
\bbE\Bigl(\max_n\|u_{k+1}^{(n)}-\ud_{k+1}\|^2\Bigr) \le \alpha^2 \bbE
\Bigl(\max_n\|u_{k}^{(n)}-\ud_{k}\|^2\Bigr)
+{\cal O}(\gamma^2)
\]
and the result follows by Gronwall.
\end{proof}
\subsection{Gaussianized Optimal Filter}

\begin{thm}
\label{thm:gopf_acc}
Let Assumptions \ref{ass:contraction}, \ref{ass:contraction} hold and consider the GOPF with particles $\{v_k^{(n)}\}_{n=1}^N$ defined by \eqref{e:gopf_rds1} (or \eqref{e:gopf_rds2})  with data $\{\yd_k\}$
given by Assumption \ref{ass:dagger}.
It follows that there is constant $c$ such that
$$\limsup_{k \to \infty} \bbE \Bigl(\max_n\|v_{k}^{(n)}-\ud_k\|^2\Bigr) \le c\gamma^2.$$
\end{thm}

\begin{proof}
Recall the notation defined at the beginning of the proof of Theorem \ref{thm:opf_acc}. Recall also the RDS representation of the GOPF \eqref{e:gopf_rds2}
\begin{equation}\label{e:gopf_rds_acc}
\begin{aligned}
\vhat_{k+1}^{(m,n)}&=(I-KH)\psi(v_k^{(m)})
+Ky_{k+1}+\gamma \zeta_{0,k}^{(m,n)}\\
v_{k+1}^{(n)}&=\sum_{m=1}^{N} {\bbI}_{I_{k+1}^{(m)}}(r_{k+1}^{(n)})\vhat_{k+1}^{(m,n)}\;.
\end{aligned}
\end{equation} 
where we now have $\zeta_{0,k}^{(m,n)} \sim N(0,C_0)$ \iid, recalling that $C = \gamma^2C_0$.
We also have the identity
\begin{align}
\ud_{k+1}&=(I-KH)\psi(\ud_{k})+KH\psi(\ud_{k})+r\gamma\xid_k\;,
\end{align}
where $\ztd_{0,k} \sim N(0,C_0)$ i.i.d. Subtracting, we obtain
\begin{equation}
\label{eq:Gee}
\vhat_{k+1}^{(m,n)}-\ud_{k+1}=(I-KH)\bigl(\psi(v_{k}^{(m)})-\psi(\ud_{k})\bigr)+
\gamma \iota_k^{(n)}
\end{equation}
where $\iota_k^{(n)}:= (K\etd_{k+1}+\ztd_{0,k}-r\xid_k).$
Note that
\begin{equation}
\label{eq:Gnp2}
\ud_{k+1}=\sum_{m=1}^{N} {\bbI}_{I_{k+1}^{(m)}}(r_{k+1}^{(n)})\ud_{k+1}
\end{equation} 
so that, defining
\[
e_k^{(n)}=v_k^{(n)}-\ud_k, \quad \he_k^{(m,n)}=\vhat_k^{(m,n)}-\ud_k
\]
we have from \eqref{e:gopf_rds_acc}, \eqref{eq:Gee} and \eqref{eq:Gnp2}
\[
e_{k+1}^{(n)}=\sum_{m=1}^{N} {\bbI}_{I_{k+1}^{(m)}}(r_{k+1}^{(n)})\he_{k+1}^{(m,n)}
.\]
Thus
\[
\max_n \|e_{k+1}^{(n)}\|^2 \le \max_m \|\he_{k+1}^{(m,n)}\|^2\;,
\]
where the norm is the one in which we have a contraction.
Using \eqref{eq:Gee}, using the Lipschitz property of $(I-KH)\psi(\cdot)$, 
taking expectations and using independence, gives
\[
\bbE\Bigl(\max_n\|v_{k+1}^{(n)}-\ud_{k+1}\|^2\Bigr) \le \alpha^2 \bbE
\Bigl(\max_n\|v_{k}^{(n)}-\ud_{k}\|^2\Bigr)
+{\cal O}(\gamma^2)\;.
\]
The result follows by Gronwall.
\end{proof}

\section{Consistency in the Large Particle Limit}\label{s:consist}

In this section we state and prove consistency results for the BPF, OPF and 
GOPF introduced in Section \ref{s:particles}, in a simple unified framework.
For the BPF the result is well known but we reproduce it here as the prove
serves as an ideological template for the more complicated proofs to follow; furthermore we present the clean proof given in 
\cite{vanhandel15} (see also \cite[Chapter 4]{LSZ15}) as this particular
approach to the result generalizes naturally to the OPF and GOPF.
We also note that the more general analysis of the consistency of the
auxiliary particle filter in \cite{johansen2008note} implies the result that
we prove here about the GOPF.

\subsection{Bootstrap Particle Filter}

In the following, we let $f_{k+1} : \CX \to\reals $ be any function with $f_{k+1}(u_{k+1}) \propto P(y_{k+1} | u_{k+1})$; any proportionality constant will suffice, but
the normalization constant is of course natural. As in previous sections, 
we let $\mu_k$ denote the filtering distribution. The following theorem is
stated and then proved through a sequence of lemmas in the remainder of
the subsection.
 
\begin{thm}\label{thm:bpf_consist}
Let $\bpfw^N_k,\bpfe^N_k$ be the BPFs defined by \eqref{e:bpfw}, \eqref{e:bpfe} respectively, and suppose that there exists a constant $\kappa \in (0,1]$ such that  
\begin{equ}\label{ass:bpf_kappa}
\kappa  \leq f_{k+1}(u_{k+1})  \leq \kappa^{-1}   
\end{equ}
for all $u_{k+1}\in\CX$, $y_{k+1} \in \CY$ and $k \in \{0, \dots, K-1\}$. Then we have
\begin{equ}\label{e:consist_bpf1}
d (\bpfw_{K}^N , \mu_K)  \leq  \sum_{k=1}^K (2\kappa^{-2})^k N^{-1/2}
\end{equ}
and
\begin{equ}\label{e:consist_bpf2}
d (\bpfe_K^N , \mu_K) \leq  \sum_{k=0}^K (2\kappa^{-2})^k N^{-1/2}
\end{equ}
for all $K,N\geq 1$. 
\end{thm}

\begin{rmk}
Note that the constant $\kappa^{-2}$ appearing in the estimates above arises as the ratio of the upper and lower bounds in \eqref{ass:bpf_kappa}. In particular, we cannot optimize $\kappa$ by choosing a different proportionality constant for $f_{k+1}$.  
$\qed$
\end{rmk}

\par
Recall formulation \eqref{e:bpf_rec2} of the iteration for the filtering
distribution. In terms of understanding the approximation properties
of the BPF, the key observation is that the measures $\{\bpfw_k^N\}_{k\geq 0}$ satisfy the recursion
\begin{equ}\label{e:bpf_rec1}
\bpfw_{k+1}^N = L_{k+1} S^N P \bpfw_k^N \quad \;\quad \bpfw^N_0 = \mu_0.
\end{equ} 
where $P: \CM(\CX) \to \CM(\CX)$ is the Markov semigroup and, as defined in Section \ref{s:notation}, $S^N : \CM(\CX) \to \CM(\CX)$ 
is the sampling operator. The convergence of the measures is quantified by the metric on random elements of $\CM(\CX)$ defined by 
\[
d(\mu,\nu) = {\rm sup}_{|f|_{\infty} \leq 1} \sqrt{ \bbE^{\omega} |\mu(f)-\nu(f)|^2 }\;,
\]
where, in our setting,  $\bbE^{\omega}$ will always denote expectation with respect to the randomness in the 
sampling operator $S^N$; this metric reduces to twice the total
variation  metric,
used in studying ergodicity, when the measures are not random. The main ingredients for the proof are the following three estimates 
for the operators $P, S^N$ and $L_{k+1}$
with respect to the metric $d$. 

\begin{lemma}\label{lem:bpf_est}\label{lem:bpf_est1}
We have the following 
\begin{enumerate}
\item  $\sup_{\nu \in \CM(\CX)} d(S^N \nu , \nu ) \leq N^{-1/2}$.
\item  $d(P\mu, P \nu) \leq d(\mu,\nu)$ for all $\mu,\nu \in \CM(\CX)$.  
\end{enumerate}
\end{lemma}

\begin{proof} See \cite[Lemma 4.7, Lemma 4.8]{LSZ15}.
\end{proof}

We state the following Lemma in a slightly more general form than necessary for the BPF, as it will be applied in different contexts for the optimal particle filters.

\begin{lemma}\label{lem:gk}\label{lem:bpf_est2}\label{lem:fk}
Let $\CZ$ be a finite dimensional Euclidean space. Suppose that $g_{k+1} : \CZ \to [0,\infty)$ is bounded and that there exists $\kappa \in(0,1]$ such that 
\begin{equ}\label{e:gk_kappa} 
\kappa \leq g_{k+1}(u) \leq \kappa^{-1}
\end{equ}
for all $u \in\CZ$ and define $G_{k+1} : \CM(\CZ)\to\CM(\CZ)$ by $G_{k+1}(\nu)(\vphi) = \nu (g_{k+1} \vphi) / \nu(g_{k+1})$. Then 
\[
d(G_{k+1} \mu, G_{k+1} \nu ) \leq (2\kappa^{-2}) d(\mu,\nu)
\]
 for all $\mu,\nu \in \CM(\CZ)$. 
\end{lemma}

\begin{proof} See \cite[Lemma 4.9]{LSZ15}. 
\end{proof}

We can now prove the consistency result. 
\begin{proof}[Proof of Theorem \ref{thm:bpf_consist}]
First note that, taking $\CZ = \CX$ and $g_{k+1} = f_{k+1}$ in Lemma \ref{lem:gk}, we obtain $G_{k+1} \nu = L_{k+1} \nu$. Thus, by \eqref{ass:bpf_kappa}, it follows that $d(L_{k+1}\mu,L_{k+1}\nu) \leq (2\kappa^{-2})d(\mu,\nu)$ for all $\mu,\nu \in \CM(\CX)$. Combining this fact with the recursions given in \eqref{e:bpf_rec1}, \eqref{e:bpf_rec2} and the estimates given in Lemmas \ref{lem:bpf_est} we have
\begin{equs}
d (\bpfw_{k+1}^N , \mu_{k+1}) &=  d (L_{k+1}S^N P \bpfw_{k}^N ,  L_{k+1}P \mu_{k})\\ 
& \leq 2\kappa^{-2} d ( S^N P\bpfw_{k}^N ,   P \mu_{k}) \\
& \leq 2\kappa^{-2} \left( d ( S^N P \bpfw_{k}^N ,   P \bpfw_{k}^N) + d (P \bpfw_{k}^N ,   P \mu_{k})\right) \\
& \leq 2\kappa^{-2} N^{-1/2} + 2\kappa^{-2} d (\bpfw_{k}^N ,   \mu_{k})\;.
\end{equs}
And since $\bpfw^N_0 = \mu_0$, we obtain \eqref{e:consist_bpf1} by induction. Moreover, since $\rho_{k} = S^N \bpfw^N_{k}$ 
\begin{equ}
d(\rho_{k} , \mu_k ) = d (S^N \bpfw_k^N , \mu_k) \leq d (S^N \bpfw_k^N , \bpfw_k^N) + d (\bpfw_k^N , \mu_k) 
\end{equ}
and \eqref{e:consist_bpf2} follows. 
\end{proof}

\subsection{Sequential Importance Resampler}

In this section we will apply the above strategy to prove the corresponding consistency result for the OPF. Instead of restricting to the OPF, we will obtain results for the sequential importance resampler (SIR), for which the OPF is a special case. See \cite[sections II, III]{doucet2000sequential} for background
in sequential importance sampling, and on the use of resampling.
As with the OPF, the SIR is an empirical measure 
\begin{equ}\label{e:sire}
\opfe_k^N =  \sum_{n=1}^N \frac{1}{N}\delta_{u_k^{(n)}}\;.
\end{equ}
We will abuse notation slightly by keeping the same notation for the OPF and the SIR. The particle positions are drawn from a proposal distribution $\pi(u_{k+1} | u_k , y_{k+1})$ and re-weighted accordingly. As usual, the positions are initialized with $u_0^{(n)} \sim \mu_0$ and updated by
 \begin{enumerate}
\item Draw $\uhat_{k+1}^{(n)}$ from $\pi(u_{k+1}| u_k^{(n)}, y_{k+1})$
\item Define the weights $w_{k+1}^{(n)}$ for $n=1,\dots,N$ by
\begin{equ}
w_{k+1}^{(n),*} = \frac{P(y_{k+1} | \uhat_{k+1}^{(n)})P(\uhat_{k+1}^{(n)} | u_k^{(n)})}{\pi(\uhat_{k+1}^{(n)} | u_k^{(n)}, y_{k+1})} \quad ,\quad w_{k+1}^{(n)} = \frac{w_{k+1}^{(n),*}}{\sum_{m=1}^N w_{k+1}^{(m),*}} \;.
\end{equ}
\item Draw $u_{k+1}^{(n)}$ from $\{\uhat_{k+1}^{(m)}\}_{m=1}^N$ with weights $\{w_{k+1}^{(m)}\}_{m=1}^N$.
\end{enumerate}
Thus, if we take the proposal to be $\pi(u_{k+1} | u_k,y_{k+1}) = P(u_{k+1} | u_k,y_{k+1})$ then we obtain the OPF \eqref{e:opfe}. Without being more specific about the proposal $\pi$, it is not possible to represent the SIR as a random dynamical system in general.

Precisely as with the OPF, for the SIR we define the related filter 
\begin{equ}\label{e:sirw}
\sirw_k^N = \sum_{n=1}^N w_{k}^{(n)} \delta_{\uhat_k^{(n)}}
\end{equ}
with $\sirw_0^N = \mu_0$ and note the important identity $\sire^N_k = S^N \sirw^N_k$. 
The following theorem, and corollary, are proved in the remainder of the
subsection, through a sequence of lemmas.

\begin{thm}\label{thm:sir_consist}
Let $\sirw^N,\sire^N$ be the SIR filters defined by \eqref{e:sirw}, \eqref{e:sire} respectively, with proposal distribution $\pi$. Suppose that there exists $f_{k+1} : \CX\times\CX\to \reals$  with  
\begin{equ}\label{e:gk_propto}
f_{k+1}(u_{k+1},u_k) \propto \frac{P(y_{k+1} | u_{k+1}) P(u_{k+1} | u_k)}{\pi(u_{k+1} | u_k, y_{k+1})}    
\end{equ}
and satisfying 
\begin{equ}\label{e:gk_kappa_joint}
\kappa \leq f_{k+1} (u_{k+1},u_k) \leq \kappa^{-1}
\end{equ}
for all $u_{k+1},u_k \in \CX$, $k \in \{0, \dots, K-1\}$ and some $\kappa \in (0,1]$. Then we have
\begin{equ}\label{e:consist_sir1}
d (\sirw_{K}^N , \mu_K)  \leq  \sum_{k=1}^K (2\kappa^{-2})^k N^{-1/2}
\end{equ}
and
\begin{equ}\label{e:consist_sir2}
d (\sire_K^N , \mu_K) \leq  \sum_{k=0}^K (2\kappa^{-2})^k N^{-1/2}
\end{equ}
for all $K,N\geq 1$. 
\end{thm}

\begin{rmk}
As for the boostrap particle filter, the appearance of $\kappa^{-2}$ 
reflects the ratio of the upper and lower bounds in \eqref{e:gk_propto};
hence there is nothing to be gained from optimizing over the
constant of proportionality. If we let
\[
f_{k+1}(u_{k+1},u_k) =  \frac{P(y_{k+1} | u_{k+1}) P(u_{k+1} | u_k)}{\pi(u_{k+1} | u_k, y_{k+1})}
\]
 then the estimate \eqref{e:gk_kappa_joint} is equivalent to
\[
\kappa \pi (u_{k+1} | u_k, y_{k+1}) \leq P(y_{k+1} | u_{k+1}) P(u_{k+1} | u_k) \leq \kappa^{-1} \pi (u_{k+1} | u_k, y_{k+1})\;.
\]
This can thus be interpreted as a quantification of equivalence between 
measures $\pi$ and the optimal proposal $P(y_{k+1} | u_{k+1}) P(u_{k+1} | u_k)$.$\qed$ 
\end{rmk}

\begin{rmk}
It is important to note that Assumption \ref{ass:cond_gauss} on the dynamics-observation model is not required by Theorem \ref{thm:sir_consist}. However
Assumption \ref{ass:cond_gauss} can be used to ensure that \eqref{e:gk_kappa_joint} holds. This observation leads to the following
corollary.
$\qed$
\end{rmk}

\begin{cor}\label{cor:consist_opf}
Let $\sirw^N,\sire^N$ be the OPFs defined in \eqref{e:opfw}, \eqref{e:opfe} respectively and satisfying Assumption \ref{ass:cond_gauss}. Then 
there is $\kappa= \kappa(Y_K)$ such that we have
\begin{equ}\label{e:consist_sir1}
d (\opfw_{K}^N , \mu_K)  \leq  \sum_{k=1}^K (2\kappa^{-2})^k N^{-1/2}
\end{equ}
and
\begin{equ}\label{e:consist_sir2}
d (\opfe_K^N , \mu_K) \leq \sum_{k=0}^K (2\kappa^{-2})^k N^{-1/2}
\end{equ}
for all $K,N\geq 1$ and where $\kappa^{-1} = \exp \left(  \max_{0\leq j \leq K-1}|y_{j+1}|^2  + \sup_{v} |H \psi(v)|_S^2\right)$.
\end{cor}

Although similar to the argument for the BPF, the recursion argument for the SIR is necessarily more complicated than that for the BPF, as the weights $w_{k+1}^{(n)}$ can potentially depend on both $u_{k+1}^{(n)}$ and $u_k^{(n)}$. This suggests that we must build a recursion which updates measures on a joint space $(u_{k+1},u_{k}) \in \CX \times \CX$. This would also be necessary if we restricted our attention to the OPF, as the weights are defined using $u_k^{(n)}$ and not the particle positions $\uhat_{k+1}^{(n)}$ after the proposal.   
\par
The recursion is defined using the following three operators. 
\begin{enumerate}
\item  First $P^\pi_{k+1}$ maps probability measures on $\CX$ to probability measures on $\CX \times \CX$ by
\begin{equ}
P^\pi_{k+1} \mu (A) = \int \!\!\int_{A} \pi(u_{k+1} | u_{k} , y_{k+1}) 
\mu(du_k) du_{k+1}  
\end{equ}
where $A$ is a measurable subset of $\CX \times \CX$.  
 \item The reweighting operator $L^\pi_{k+1}$ maps probability measures on 
$\CX \times \CX$ to probability measures on $\CX \times \CX$ and is defined by 
\begin{equ}
L^\pi_{k+1} Q (A) = Z^{-1} \int \!\!\int_{A} w_{k+1}(u_{k+1} , u_k) Q(du_{k+1},du_k) 
\end{equ}
where $Z$ is the normalization constant of the resulting measure. The weight function is given by
\begin{equ}
w_{k+1} (u_{k+1},u_k) = \frac{P(y_{k+1}| u_{k+1})P(u_{k+1}| u_{k})}{\pi(u_{k+1}| u_k , y_{k+1})} \;.
\end{equ}
\item Finally, $M$ maps probability measures on $\CX \times \CX$
into probability measures on $\CX$ via marginalization
onto the first component:
$$M Q (B) = \int \int_{B \times \CX} Q(du_{k+1},du_{k})$$.
\end{enumerate}
It is easy to see that the posterior $\mu_k$ satisfies a natural recursion in terms of these operators. 
\begin{lemma}\label{lem:sir_rec1}
$\mu_{k+1} = M L^\pi_{k+1} P^\pi_{k+1} \mu_k$
\end{lemma} 
\begin{proof}
Let $P(u_k | Y_k)$ denote the density of $\mu_k$, then $P^\pi_{k+1} \mu_k$ is a measure on $\CX\times\CX$ with density 
\begin{equ}
\pi(u_{k+1} | u_k , y_{k+1}) P(u_k |Y_k ) \;.
\end{equ}
And $L^\pi_{k+1} P^\pi_{k+1} \mu_k$ is a measure on $\CX \times \CX$ with density
\begin{equ}
Z^{-1} w_{k+1}(u_{k+1},u_k)\pi(u_{k+1} | u_k , y_{k+1}) P(u_k |Y_k )  = Z^{-1} P(y_{k+1} | u_{k+1})P  (u_{k+1} | u_k ) P(u_k | Y_k) \;.
\end{equ}
Finally, $M L^\pi_{k+1} P^\pi_{k+1} \mu_k$ is a measure on $\CX$ with density
\begin{equ}\label{e:post_with_Z}
\int_\CX Z^{-1} P(y_{k+1} | u_{k+1}) P(u_{k+1} | u_k) P (u_{k} | Y_k)du_k = Z^{-1} P(y_{k+1} | u_{k+1}) P (u_{k+1} | Y_k)\;. 
\end{equ}
Similarly, for the normalization factor, we have 
\begin{equs}
Z = \int \int_{\CX\times \CX} P(y_{k+1} | u_{k+1})P (u_{k+1} | u_k ) P(u_k | Y_k) du_k du_{k+1} = \int_\CX P(y_{k+1} | u_{k+1}) P (u_{k+1} | Y_k) du_k 
\end{equs}
and thus by Bayes' formula \eqref{e:post_with_Z} is equal to $P(u_{k+1} | Y_{k+1})$ as required. 
\end{proof}
We now show that an associated recursion is satisfied by the SIR filter $\sirw^N$.  

\begin{lemma}\label{lem:sir_rec2}
Let $\sirw^N$ be the SIR filter given by \eqref{e:sirw}, then 
\begin{equ}
\sirw_{k+1}^N =  M L^\pi_{k+1}  S^N P^\pi_{k+1} \sirw_k^N
\end{equ}
for all $k \geq 0$ and $N\geq1$, where $S^N$ denotes the sampling operator acting on $\CM(\CX\times \CX)$.  
\end{lemma}

\begin{proof}
By definition, $\sirw_k^N = \sum_{n=1}^N w_k^{(n)} \delta_{\uhat_k^{(n)}}$ so that $P^\pi_{k+1} \sirw_k^N \in \CM (\CX\times \CX) $ with density
\begin{equ}
\sum_{n=1}^N w_k^{(n)} \pi (u_{k+1} | \uhat_k^{(n)},y_{k+1}) \delta(u_{k} -\uhat_{k}^{(n)} )\;.
\end{equ}
Note that a sample $U \sim P^\pi_{k+1} \sirw_k^N $ is a pair $(\uhat_{k+1}^{(n)},u_k^{(n)})$ obtained as follows: first draw a sample $u_k^{(n)}$ from $\{\uhat_k^{(n)}\}_{n=1}^N$ with weights $\{w_k^{(n)}\}_{n=1}^N$ and then draw sample
$\uhat_{k+1}^{(n)}$ from $\pi (u_{k+1} | u_{k}^{(n)},y_{k+1})$. Thus, by definition of the $\uhat_{k+1}^{(n)}$ sequence we see that $S^N P^\pi_{k+1} \sirw_k^N$ has density 
\begin{equ}
\frac{1}{N}\sum_{n=1}^N \delta(u_{k+1 } - \uhat_{k+1}^{(n)})\delta(u_{k } - u_{k}^{(n)})\;.
\end{equ}
It follows that $L^\pi_{k+1} S^N P^\pi_{k+1} \murec_k^N $ has density 
\begin{equ}
\sum_{n=1}^N Z^{-1} w_{k+1}(\uhat_{k+1}^{(n)}, u_k^{(n)})\delta(u_{k+1 } - \uhat_{k+1}^{(n)})\delta(u_{k } - u_{k}^{(n)})
\end{equ} 
and $M L^\pi_{k+1} S^N P^\pi_{k+1} \murec_k^N$ has density 
\begin{equ}
\sum_{n=1}^N Z^{-1} w_{k+1}(\uhat_{k+1}^{(n)}, u_k^{(n)})\delta(u_{k+1 } - \uhat_{k+1}^{(n)})\;.
\end{equ} 
Lastly, the normalization factor is given by 
\begin{equs}
Z = \int \int_{\CX\times \CX} \sum_{n=1}^N w_{k+1}(\uhat_{k+1}^{(n)}, u_k^{(n)})\delta(u_{k+1 } - \uhat_{k+1}^{(n)})\delta(u_{k } - u_{k}^{(n)}) du_{k} du_{k+1} = \sum_{n=1}^N w_{k+1}(\uhat_{k+1}^{(n)}, u_k^{(n)})\;,
\end{equs}
so that $Z^{-1} w_{k+1}(\uhat_{k+1}^{(n)}, u_k^{(n)}) = w_{k+1}^{(n)}$ and we obtain the result.
\end{proof}

In the final step before proving Theorem \ref{thm:sir_consist}, we state some simple properties for the operators appearing in the recursions. Note that these are similar but not (all) immediately implied by the corresponding results for the BPF, Lemma \ref{lem:bpf_est1}. 

\begin{lemma}\label{lem:sir_est1}
We have the following simple estimates:
\begin{enumerate}
\item $d (M\nu ,M \mu) \leq d (\nu,\mu)$.
\item $d (P^\pi_{k+1}\nu ,P^\pi_{k+1}\mu) \leq d (\nu,\mu)$
\item $\sup_{\nu \in \CM(\CX \times \CX)}d (S^N \nu , \nu) \leq N^{-1/2}$
\end{enumerate}
\end{lemma}

\begin{proof} Let $\ftilde(x,y) = f(x)$ and let $g(x,y)$ denote
an arbitrary function. Then
\begin{equs}
M\nu (f) - M\mu(f) = \nu(\ftilde) - \mu(\ftilde) 
\end{equs}
The first inequality follows immediately from taking supremum over all 
$|f|\leq 1$, which is necessarily smaller than the supremum of 
$\nu(g) - \mu(g)$ over all $|g| \leq 1 $.
\par
We also have 
\begin{equ}
P^\pi_{k+1} \nu (g) - P^\pi_{k+1} \mu(g) = \nu(g^\pi) - \mu(g^\pi)
\end{equ}
where $g^\pi (u_k) = \int g(u_{k+1},u_k) \pi(u_{k+1} | u_k, Y_{k+1}) du_{k+1}$. And since $|g^\pi|_{\infty}\leq 1$, the second inequality follows. The third inequality is proven in \cite[Lemma 4.7]{LSZ15}, simply replacing $\CX$ with $\CX \times \CX$.  
\end{proof}

We can now proceed with the main result. 

\begin{proof}[Proof of Theorem \ref{thm:sir_consist}]
In the context of Lemma \ref{lem:fk}, take $\CZ = \CX \times \CX$ and $g_{k+1} = f_{k+1}$, it follows that $G_{k+1 }\nu = L_{k+1}^\pi \nu$.Indeed, for any $\varphi : \CX \times \CX \to \reals$ and with $g_{k+1} = Z^{-1} w_{k+1}$ we have 
\[
G_{k+1} \nu (\varphi) = \frac{\nu (g_{k+1} \varphi)}{\nu(g_{k+1})} = \frac{\nu (w_{k+1} \varphi)}{\nu(w_{k+1})} = L_{k+1}^\pi \nu (\varphi)\;.
\]
By Assumption \ref{e:gk_kappa_joint}, we therefore obtain from Lemma \ref{lem:fk} that
\[
d (L_{k+1}^\pi \mu , L_{k+1}^\pi \nu) \leq (2\kappa^{-2})d(\mu,\nu)
\]
for all $\mu,\nu\in\CM(\CX\times\CX)$.  
\par
Thus, using the recursions given in Lemmas \ref{lem:sir_rec1}, \ref{lem:sir_rec2} and the estimates given in Lemma \ref{lem:sir_est1}, we obtain
\begin{equs}
d (\sirw_{k+1}^N , \mu_{k+1}) &=  d (M L_{k+1}^\pi S^N P_{k+1}^\pi\sirw_{k}^N , M L_{k+1}^\pi  P_{k+1}^\pi\mu_{k})\\ &\leq d ( L_{k+1}^\pi S^N P_{k+1}^\pi\sirw_{k}^N ,  L_{k+1}^\pi  P_{k+1}^\pi\mu_{k}) \\
& \leq 2\kappa^{-2}d ( S^N P_{k+1}^\pi\sirw_{k}^N ,   P_{k+1}^\pi\mu_{k}) \\
& \leq 2\kappa^{-2} \left( d ( S^N P_{k+1}^\pi\sirw_{k}^N ,   P_{k+1}^\pi\sirw_{k}^N) + d (P_{k+1}^\pi\sirw_{k}^N ,   P_{k+1}^\pi\mu_{k})\right) \\
& \leq 2\kappa^{-2} N^{-1/2} + 2\kappa^{-2} d (\sirw_{k}^N ,   \mu_{k})\;.
\end{equs}
and since $\sirw^N_0 = \mu_0$, we obtain \eqref{e:consist_sir1} by induction. Moreover, since $\sire^N_{k} = S^N \sirw^N_{k}$ 
\begin{equ}
d(\sire_{k}^N , \mu_k ) = d (S^N \sirw_k^N , \mu_k) \leq d (S^N \sirw_k^N , \sirw_k^N) + d (\sirw_k^N , \mu_k) 
\end{equ}
and \eqref{e:consist_sir2} follows. 
\end{proof}

The corollary follows immediately. 
\begin{proof}[Proof of Corollary \ref{cor:consist_opf}]
For the OPF we have 
\begin{equ}
\pi(u_{k+1} | u_k , y_{k+1}) = P(u_{k+1} | u_k , y_{k+1}) = \frac{P(y_{k+1} | u_{k+1}) P(u_{k+1} | u_k)}{P(y_{k+1} | u_k)}
\end{equ}
where we have applied Bayes formula in the final equality. But under Assumption \ref{ass:cond_gauss} we have that 
\begin{equ}
P(y_{k+1} | u_k ) = Z_S^{-1} \exp\left(-\frac{1}{2} |y_{k+1} - H\psi(u_k)|_S^2\right)\;.
\end{equ}
Thus we define $f_{k+1}$ by
\begin{equ}
f_{k+1} (u_{k+1},u_k) = Z_S\frac{P(y_{k+1} | u_{k+1}) P(u_{k+1} | u_k)}{\pi (u_{k+1} | u_k , y_{k+1}) }= \exp\left(-\frac{1}{2} |y_{k+1} - H\psi(u_k)|_S^2\right)
\end{equ}
and hence \eqref{e:gk_kappa_joint} holds with $\kappa^{-1} = \exp \left(  \max_{0\leq j \leq K-1}|y_{j+1}|^2  +  \sup_{v} |H \psi(v)|_S^2\right)$, which, for each $Y_{k}$, is finite by Assumption \ref{ass:cond_gauss}. The result follows from Theorem \ref{thm:sir_consist}. 
\end{proof}

\subsection{Gaussianized Optimal Particle Filter} 
In this section we derive the consistency result for the GOPF.

\begin{thm}
Let $\gopf^N$ be the GOPF defined by \eqref{e:gopf} and let
Assumption \ref{ass:cond_gauss} hold. Then there is 
$\kappa=\kappa(Y_K)$ such that
\begin{equ}
d (\gopf_{K}^N , \mu_K)  \leq  \sum_{k=0}^K (2\kappa^{-2})^k N^{-1/2}
\end{equ}
for all $K,N\geq 1$, where $\kappa^{-1} = \exp \left(  \max_{0\leq j \leq K-1}|y_{j+1}|^2   +\sup_{v} |H \psi(v)|_S^2\right)$.
\end{thm}
For the GOPF, the consistency proof uses the same strategy, but turns out to be much more straightforward. First note that the decomposition of the filtering distribution given in \eqref{e:opf_bayes} gives the recursion formula
\begin{equ}\label{e:gopf_rec1}
\mu_{k+1} = Q_{k+1} K_{k+1} \mu_k 
\end{equ}
where $K_{k+1} : \CM(\CX) \to \CM(\CX)$ is defined by
\begin{equ}
K_{k+1} \mu (A) = Z^{-1}\int_A P(y_{k+1} | u_k) \mu(du_k)
\end{equ}
for all measurable $A \subset \CX$ where $Z$ is the normalization constant, and $Q_{k+1} : \CM(\CX) \to \CM(\CX)$ is the Markov semigroup with kernel $P(u_{k+1} | u_k, y_{k+1})$. 
\par
Moreover, we have the following recursion for the GOPF.
Let $\gopf_k^N = \frac{1}{N}\sum_{n=1}^N \delta_{v_k^{(n)}}.$

\begin{lemma}\label{lem:gopf_rec2} The GOPF $\gopf^N_k$ satsfies the recursion
\begin{equ}\label{e:gopf_rec2}
\gopf^N_{k+1} = S^N Q_{k+1} K_{k+1} \gopf^N_k
\end{equ}
with $\gopf_0^N = S^N\mu_0$\;.
\end{lemma} 
\begin{proof}
Note that $K_{k+1} \in \CM(\CX)$ with density 
\begin{equ}
\sum_{n=1}^N Z^{-1} P(y_{k+1} | v_k^{(n)})\delta(v_{k} - v_{k}^{(n)})\;.
\end{equ}
The normalization constant is given by
\begin{equ}
Z = \int_\CX \sum_{n=1}^N  P(y_{k+1} | v_k^{(n)})\delta(v_{k} - v_{k}^{(n)}) dv_{k}  = \sum_{n=1}^N  P(y_{k+1} | v_k^{(n)})
\end{equ}
and thus $Z^{-1} P(y_{k+1} | v_k^{(n)}) = w_{k+1}^{(n)}$. We then have $Q_{k+1} K_{k+1}\gopf_k^N \in\CM(\CX)$ with density
\begin{equ}
\sum_{n=1}^N w_{k+1}^{(n)} P(v_{k+1} | v^{(n)}_k , y_{k+1}) \;.
\end{equ}
To draw a sample $v_{k+1}^{(n)}$ from this mixture model, we draw $\vtilde_k^{(n)}$ from $\{v_k^{(m)}\}_{m=1}^N $ with weights $\{w_{k+1}^{(m)}\}_{m=1}^N$ and then draw $v_{k+1}^{(n)}$ from $P(v_{k+1} | \vtilde_k^{(n)},y_{k+1})$. It follows that $S^N Q_{k+1} K_{k+1} \gopf_k^N = \gopf_{k+1}^N$. 

\end{proof}


\begin{proof} 
If we let 
\[
g_{k+1} (v_{k}) :=  Z_S P(y_{k+1} | v_k) = \exp\left(-\frac{1}{2} |y_{k+1} - H\psi(v_k)|_S^2\right)\;,
\]
then $g_{k+1}$ satisfies the assumptions of Lemma \ref{lem:fk} with 
$$\kappa^{-1}=\exp \left(  \max_{0\leq j \leq K-1}|y_{j+1}|^2  + \sup_{v} |H \psi(v)|_S^2\right).$$  In particular, since $G_{k+1}\nu = K_{k+1} \nu$, it follows from Lemma \ref{lem:fk} that 
\[
d(K_{k+1}\mu, K_{k+1}\nu) \leq (2\kappa^{-2})d(\mu,\nu)
\]
for all $\mu,\nu \in \CM(\CX)$. 
\par
Using the recursions \eqref{e:gopf_rec1},\eqref{e:gopf_rec2} and the estimates from Lemma \ref{lem:bpf_est1}, we see that  
\begin{equs}
d(\gopf_{k+1}^N, \mu_{k+1} ) &= d(S^N Q_{k+1} K_{k+1} \gopf_k^N , Q_{k+1} K_{k+1} \mu_k) \\
& \leq d(S^N Q_{k+1} K_{k+1} \gopf_k^N , Q_{k+1} K_{k+1} \gopf_k^N) + d(Q_{k+1} K_{k+1} \gopf_k^N , Q_{k+1} K_{k+1} \mu_k)\\ 
& \leq N^{-1/2} + d( K_{k+1} \gopf_k^N ,  K_{k+1} \mu_k)\\
& \leq N^{-1/2} + 2\kappa^{-2} d(\gopf_k^N,\mu_k^N)
\end{equs}
by induction, we obtain
\begin{equ}
d(\gopf_{k+1}^N , \mu_{k+1}^N) \leq \sum_{j=0}^k (2\kappa^{-2})^j N^{-1/2} + (2\kappa^{-2})^{k+1} d(\gopf_{0}^N, \mu_0)\;.
\end{equ}
And the result follows from the fact $d(\gopf_{0}^N, \mu_0) = d(S^N \mu_0 , \mu_0) \leq N^{-1/2}$. 
\end{proof}

\vspace{0.1in}
\noindent{\bf Acknowledgements.} DK was supported as an NYU-Courant instructor.
The work of AMS was supported by EPSRC (EQUIPS Program Grant), DARPA 
(contract W911NF-15-2-0121) and ONR (award N00014-17-1-2079). The authors are grateful to an anonymous referee for comments on an earlier version of this paper
which have led to a correction of the presentation in subsection \ref{ssec:CE},
and pointers to relevant literature that were previously omitted; they
are also grateful to Jonathan Mattingly for discussions relating to ergodicity
of nonlinear Markov processes. The authors also highlight, on the occasion of his $70^{th}$ birthday,
the inspirational work of Andy Majda, in applied
mathematics in general, and in the area of data assimilation and filtering
in particular.

\bibliographystyle{plain}
\bibliography{filter_new}

\end{document}